\documentclass[aos]{imsart}

\RequirePackage{amsthm,amsmath,amsfonts,amssymb}
\RequirePackage[numbers]{natbib}
\usepackage[utf8]{inputenc}
\usepackage{amssymb, mathtools}
\usepackage{graphicx}
\usepackage{float} % for [H] figure placement
\usepackage{bm} % 新增宏包，用于数学黑体 \mathbf
\usepackage{float} % provides [H]
\usepackage{xcolor}  
\newtheorem{theorem}{Theorem}
\newtheorem{lemma}{Lemma}
\newtheorem{proposition}{Proposition}
\newtheorem{remark}{Remark}
\newtheorem{definition}{Definition}

\newcommand{\R}{\mathbb{R}}
\newcommand{\E}{\mathbb{E}}
\newcommand{\Var}{\mathrm{Var}}
\newcommand{\Cov}{\mathrm{Cov}}
\newcommand{\Corr}{\mathrm{Corr}}
\newcommand{\Tr}{\mathrm{tr}}
\newcommand{\tr}{\Tr}
\newcommand{\op}{\mathrm{op}}
\newcommand{\norm}[1]{\left\lVert #1\right\rVert}
\newcommand{\opnorm}[1]{\norm{#1}_{\mathrm{op}}}
\newcommand{\Fnorm}[1]{\norm{#1}_{F}}

\startlocaldefs
\endlocaldefs

\begin{document}

\begin{frontmatter}
\title{Local Asymptotic Normality for Mixed Fractional Brownian Motion under High-Frequency Observation}
%\title{A sample article title with some additional note\thanksref{T1}}
\runtitle{LAN PROPERTY FOR MFBM UNDER HIGH-FREQUENCY OBSERVATION}
%\thankstext{T1}{A sample of additional note to the title.}

\begin{aug}
%%%%%%%%%%%%%%%%%%%%%%%%%%%%%%%%%%%%%%%%%%%%%%%
%% Only one address is permitted per author. %%
%% Only division, organization and e-mail is %%
%% included in the address.                  %%
%% Additional information such as            %%
%% identifying the corresponding author must %%
%% be included in in the Acknowledgments     %%
%% section if necessary.                     %%
\author[A]{\fnms{Chunhao}~\snm{Cai}\ead[label=e1]{caichh9@mail.sysu.edu.cn}}
\and \author[B]{\fnms{Yiwu}~\snm{Shang}\ead[label=e2]{shanyw@mail.nankai.edu.cn}}
%%%%%%%%%%%%%%%%%%%%%%%%%%%%%%%%%%%%%%%%%%%%%%
%% Addresses                                %%
%%%%%%%%%%%%%%%%%%%%%%%%%%%%%%%%%%%%%%%%%%%%%%
\address[A]{School of Mathematics (Zhuhai), Sun Yat-Sen University\printead[presep={,\ }]{e1}}
\address[B]{ School of mathematical sciences,Nankai University\printead[presep={,\ }]{e2}}
\end{aug}

\begin{abstract}
In this paper we will consider the LAN property for both the Hurst parameter $H>3/4$ and the variance of the fractional Brownian motion plus an independent standard Brownian motion (called mixed fractional Brownian motion) with high-frequency observation. We will first remove the $H$-score linear term and orthogonalize the remainder through two non-diagonal transformations, then we can construct the CLT for the quadratic form base on $\| \cdot \|_{\mathrm{op}}/\|\cdot\|_F\to0$. At last we obtain a diagonal Gaussian LAN expansion with an explicit information matrix.  Beyond the case of $H>3/4$, we also present  that the $\| \cdot \|_{\mathrm{op}}/\|\cdot\|_F\to0$ method is also useful for the case of  $H<3/4$ and the proof will be concise compared with the Whittle translation method. We consider that this method can be applied to this type of problem, including the fractional Ornstein-Uhlenbeck model and mixed fractional O-U process.
\end{abstract}



\end{frontmatter}

\section{Introduction}\label{sec:intro}

We consider the mixed fractional Brownian motion (mfBm)
\begin{equation}\label{eq:model}
    Y_t=\sigma B_t^H+{B_t},\qquad t\ge 0,
\end{equation}
where $B^H$ is a standard fractional Brownian motion with Hurst index $H\in(3/4,1)$, {$B$} is an independent standard Brownian motion, and the parameter $\theta=(\sigma,H)$ is unknown.
We observe discrete high-frequency increments on a growing time horizon:
\[
\Delta_n=n^{-\alpha},\quad \alpha\in(0,1),\qquad \mathcal T_n:=n\Delta_n\longrightarrow\infty,
\]
\begin{equation}\label{eq:increments}
X_{n,i}:=Y_{i\Delta_n}-Y_{(i-1)\Delta_n}
=\sigma \Delta_n^H G_{n,i}+\Delta_n^{1/2}{W_{n,i}},\qquad i=1,\dots,n,  
\end{equation}
where $(G_{n,i})$ is fractional Gaussian noise (fGn) and {$(W_{n,i})$} is standard Gaussian white noise, independent of $(G_{n,i})$. In this paper we will construct the LAN property for the unknown parameter $\theta$ under this high-frequency observation. 

We now recall Le Cam's notion of local asymptotic normality (LAN), which provides a precise second-order approximation of the log-likelihood ratio under local perturbations and yields sharp minimax lower bounds. 

\begin{definition} [LAN]\label{def:LAN}
Let $\bigl(\mathbb P_\theta^n\bigr)_{\theta\in\Theta}$ be a sequence of statistical experiments indexed by $\theta\in\Theta\subset\mathbb R^k$.
We say it is \emph{locally asymptotically normal} at $\theta_0\in\Theta$ if there exist non-singular $k\times k$ matrices $\phi(n)=\phi(n,\theta_0)$ with $\|\phi(n)\|\to0$, random vectors $Z_{n,\theta_0}\in\mathbb R^k$, and a symmetric nondegenerate matrix $I(\theta_0)$ such that for every fixed $u\in\mathbb R^k$,
\begin{equation}
    \log\frac{d\mathbb P^n_{\theta_0+\phi(n)u}}{d\mathbb P^n_{\theta_0}}(X_n)
=
u^\top Z_{n,\theta_0}-\frac12 u^\top I(\theta_0)u+o_{\mathbb P_{\theta_0}}(1), \quad Z_{n,\theta_0}\Rightarrow \mathcal N\bigl(0,I(\theta_0)\bigr). \nonumber
\end{equation}
 
\end{definition}
Denote by $W^{2,k}$ the class of loss functions $\ell:\mathbb R^k\to[0,\infty)$ that are symmetric, nonnegative quasi-convex function and satisfy
\[
\lim_{\|u\|\to\infty}e^{-a\|u\|^2}\,\ell(u)=0,\qquad \forall a>0.
\]
The following is the classical local asymptotic minimax lower bound for LAN families.

\begin{theorem}\label{thm:LAN-minimax}
Assume $\bigl(\mathbb P_\theta^n\bigr)_{\theta\in\Theta}$ is LAN at $\theta_0$ with normalization matrices $\phi(n)$ and information $I(\theta_0)$.
Then for any sequence of estimators $\hat\theta_n$, any $\ell\in W^{2,k}$, and $\delta>0$,
\[
\liminf_{\delta\rightarrow\infty}\liminf_{n\to\infty}\;
\sup_{\|\theta-\theta_0\|<\delta}
\mathbb E_{\theta}\!\left[
\ell\!\left(\phi(n)^{-1}\bigl(\hat\theta_n-\theta\bigr)\right)
\right]
\ \ge\
\int_{\mathbb R^k}\ell(x)\,\varphi_{I(\theta_0)^{-1}}(x)\,dx,
\]
where $\varphi_{I(\theta_0)^{-1}}$ denotes the centered Gaussian density with covariance $I(\theta_0)^{-1}$ (restricted to the range of $I(\theta_0)$ when $I(\theta_0)$ is nonsingular).
\end{theorem}
In a continuous-noise setting Kleptsyna and Chigansky \cite{KleptsynaChigansky}  have obtained LAN-type statements under continuous observation. A natural question is whether one can simply ``transfer'' continuous-observation LAN results by replacing the continuous time horizon with the discrete one $\mathcal T_n$.

However, in the present infill scheme \eqref{eq:increments}, the observation is inherently discrete and contaminated by the Brownian component at the increment level: the small-scale behavior is dominated by the $\Delta_n^{1/2}$-term, while the long-memory component enters through a vanishing regularization factor (in our notation later $\gamma_n=\sigma^2\Delta_n^{2H-1}\to0$).
This mismatch changes both the effective Fisher information scaling and the form of the local reparametrization.

In \cite{CaiWhittle25}, the author has considered the same model when the  
long-memory singularity is mild (notably $H<3/4$), Whittle-type approximations can be used to translate Toeplitz likelihoods into spectral integrals and to establish LAN in a comparatively direct way. However, the regime $H>3/4$ corresponds to a stronger low-frequency singularity: the relevant trace functionals are driven by the near-zero behavior of the spectral density of fGn $f_H$ and its $H$-derivative and require refined trace asymptotics for \emph{triangular-array} (i.e.\ $n$-dependent) regularization.  

The most delicate phenomenon in the regime $H>3/4$ is that the (properly normalized) score vector exhibits an additional asymptotic co-linearity (see the Appendix) which is not present in the pure fGn model \cite{BrousteFukasawa18}. 
This degeneracy arises from a specific interplay between the vanishing signal-to-noise ratio and the severity of the spectral singularity. 
Unlike pure fGn, where self-similarity preserves the signal structure across scales, the mfBm signal is asymptotically buried under white noise because the fractional variance $\Delta_n^{2H}$ decays faster than the white noise variance $\Delta_n$ whenever $H>1/2$. 
However, noise dominance alone does not imply degeneracy; in the intermediate regime $1/2 < H < 3/4$, where the noise also dominates, the LAN property can be established without projection because the spectral density is square-integrable (the memory parameter $p=2H-1 < 1/2$). 
In stark contrast, when $H > 3/4$, we have $p > 1/2$, rendering the squared spectral density non-integrable at the origin. 
This strong singularity concentrates the Fisher information so heavily at zero frequency that the spectral ``shape'' variations (governed by $H$) become asymptotically indistinguishable from the ``energy'' scaling (governed by $\sigma$). 
Consequently, the standard score vectors collapse into a rank-1 subspace, necessitating our projection step to recover the orthogonal information.

The rest of the paper will be organized as follows: in Section \ref{sec: Main}
we will introduce the main results of the LAN property. Section \ref{sec: Pre} to \ref{sec:LAN} are contributed to the proof of our theorem. Then we will the simulated results in Section \ref{sec:simulation}. At last in the Appendix, some explicit formula in the Fisher information will be introduced and we will also explain why we should take two non-diagonal rate matrix for this case. Also in the Appendix we extend our $\| \cdot \|_{\mathrm{op}}/\|\cdot\|_F\to0$ method to the case $H<3/4$, this method will be more concise than the Whittle translation method presented in \cite{CaiWhittle25}. 

\section{Main Results} \label{sec: Main}
For any integrable real-valued symbol $u$ on $[-\pi,\pi]$, we define $T_n(u)$ be the $n\times n$ Toeplitz matrix with entries
\[
(T_n(u))_{jk}:=\hat u_{j-k},\qquad \hat u_k:=\frac{1}{2\pi}\int_{-\pi}^{\pi} u(\lambda)e^{-ik\lambda}\,d\lambda,\quad k\in\mathbb Z.
\]
In particular, $T_n(H)=T_n(f_H)$, where $f_H$ is the spectral density of the standard fGn under the Fourier inversion
$$
\rho_H(k)=\frac{1}{2\pi}\int_{-\pi}^{\pi} e^{ik\lambda}f_H(\lambda)\,d\lambda.
$$

Also as we know $f_H(\lambda)$ has the low-frequency asymptotic
\begin{equation}\label{eq:fH-asymp}
    f_H(\lambda)\sim c_H|\lambda|^{-(2H-1)}\qquad(\lambda\to0),
\end{equation}
for some constant $c_H>0$ depending only on $H$. We define the signal-to-noise ratio parameter
\begin{equation}\label{eq:gamma-def}
    \gamma_n=\sigma^2\Delta_n^{2H-1},
\end{equation}
then our observation ${X_n}=(X_{n,1},\dots,X_{n,n})^\top$ is a centered Gaussian vector with covariance
\begin{equation}\label{eq:Vn-def}
   V_{n}:= V_n(\theta)=\Delta_n\big(I_n+\gamma_n T_n(H)\big),\qquad \theta=(\sigma,H),
\end{equation}
where $I_n$ is the $n\times n$ identity matrix and $T_n(H)$ is the Toeplitz covariance matrix of the standard fractional Gaussian noise (fGn) with Hurst index $H\in(1/2,1)$.
We consider the increasing-horizon high-frequency sampling scheme
\begin{equation}\label{eq:sampling-scheme}
    \Delta_n=n^{-\alpha},\qquad 0<\alpha<1,
    \qquad \mathcal T_n:=n\Delta_n=n^{1-\alpha}\to\infty.
\end{equation}

\subsection{Exact score functions}\label{sec:scores}
Let $\ell_n(\theta)$ be the Gaussian log-likelihood of $X_n\sim\mathcal N(0,V_n(\theta))$:
\[
\ell_n(\theta)=-\frac{n}{2}\ln(2\pi)-\frac12\ln\det V_n(\theta)-\frac12 X_n^\top V_n(\theta)^{-1}X_n,
\]
then 
\begin{proposition}\label{prop:exact-scores}
The exact score components are centered quadratic forms
\begin{align}
S_{\sigma,n}:=\partial_\sigma \ell_n
&=\frac{\gamma_n}{\sigma\Delta_n}\Big(X_n^\top\Psi_{\sigma,n}X_n-\E[X_n^\top\Psi_{\sigma,n}X_n]\Big),\label{eq:sigma-score}\\
S_{H,n}:=\partial_H \ell_n
&=\frac{\gamma_n}{2\Delta_n}\Big(X_n^\top\Psi_{H,n}X_n-\E[X_n^\top\Psi_{H,n}X_n]\Big),\label{eq:H-score}
\end{align}
where
\begin{align*}
\Psi_{\sigma,n}&=(I_n+\gamma_nT_n(H))^{-1}\,T_n(H)\,(I_n+\gamma_nT_n(H))^{-1},\\
\Psi_{H,n}&=(I_n+\gamma_nT_n(H))^{-1}\,\big(2\ln(\Delta_n)\,T_n(H)+\dot T_n(H)\big)\,(I_n+\gamma_nT_n(H))^{-1},
\end{align*}
and $\dot T_n(H):=\partial_H T_n(H)$.
\end{proposition}

\begin{proof}
This is the standard Gaussian score identity:
\[
\partial_\theta \ell_n
=-\frac12\Tr(V_n^{-1}\partial_\theta V_n)+\frac12 X_n^\top V_n^{-1}(\partial_\theta V_n)V_n^{-1}X_n,
\]
and $\E[X_n^\top AX_n]=\Tr(AV_n)$. Using $V_n=\Delta_n(I_n+\gamma_nT_n)$ together with
$\partial_\sigma V_n=\frac{2\gamma_n\Delta_n}{\sigma}T_n$ and
$\partial_H V_n=\Delta_n\gamma_n(2\ln(\Delta_n)T_n+\dot T_n)$ yields the stated sandwich forms.
\end{proof}

\begin{theorem}\label{thm:LAN-rank2}
Let the rate matrix $M_n=M_n^{(2)}M_n^{(1)}$ defined in \eqref{eq: rate matrix}, the sequence of statistical experiments generated by the observations $X_{n}$ is Locally Asymptotically Normal at $\theta=(\sigma, H)$ for $\sigma>0,\, 3/4<H<1$. Specially for any fixed $h\in \mathbb{R}^2$, the log-likelihood ratio satisfies:
$$
\log \frac{d\mathbb{P}^n_{\theta+hr_n^{-1}}}{d\mathbb{P}_{\theta}^n}=h^\top \Xi_n-\frac12 h^\top I^\perp h+o_{\mathbb P^{n}_{\theta}}(1),
$$
where 
$$
r_n:=\sqrt{\mathcal T_n}\,(M_n^{-1})^{\!\top},
\qquad\text{equivalently}\qquad
r_n^{-\,\top}=\frac{1}{\sqrt{\mathcal T_n}}\,M_n.
$$
The random vector $\Xi_n$ converges in law to
$\mathcal N(0,I^\perp)$ and
\[
I^\perp=
\begin{pmatrix}
\frac{\sigma^2}{\pi}J_0(H,\sigma) & 0\\[2mm]
0 & \frac{\sigma^4}{4\pi}J_\perp(H,\sigma)
\end{pmatrix},
\quad
J_\perp=J_2-\frac{J_1^2}{J_0},
\]
with $p=2H-1$ and  $J_{0},J_1,J_2$ defined in \eqref{eq:J0}--\eqref{eq:J2}.
\end{theorem}

\section{Preliminaries: Score Functions}\label{sec: Pre}
\subsection{Some General Properties of fGn and Score Function}
As presented in the introduction, we deal with the score function using the standard normal quadratic formula, so first, we will reduce the it to the standard one. We present 
\begin{lemma}\label{lem:explicit-linear}
Let $A_n:=I_n+\gamma_nT_n(H)$ and $X_n =\sqrt{\Delta_n}\,A_n^{1/2}Z_{n}$. {Then we have
$$Z_{n} \sim\mathcal N(0,I_n).$$}
The $\sigma$-score admits the exact representation
\begin{equation}\label{eq:Ssigma-quad}
S_{\sigma,n} =\frac{\gamma_n}{\sigma}\Big({Z_{n}}^\top C_n {Z_{n}}-\Tr(C_n)\Big),
\end{equation}
and the $H$-score satisfies the decomposition
\begin{equation}\label{eq:Hscore-decomp}
S_{H,n} =\sigma\ln(\Delta_n)\,S_{\sigma,n}+R_{H,n},
\end{equation}
where 
\begin{equation}\label{eq:RH-def}
R_{H,n}:=\frac{\gamma_n}{2}\Big(Z_{n}^\top D_n Z_{n}-\tr(D_n)\Big),
\end{equation}
with the matrices
$$
C_n:=A_n^{-1/2}T_n(H)A_n^{-1/2},
\qquad
D_n:=A_n^{-1/2}\dot T_n(H)A_n^{-1/2}.
$$
\end{lemma}

\begin{proof}
Since $X_n \sim \mathcal{N}(0, V_n(\theta))$ and 
\[
Z_n = \frac{1}{\sqrt{\Delta_n}} A_n^{-1/2} X_n,
\]
where $A_n = I_n + \gamma_n T_n(H)$, then $Z_n \sim \mathcal{N}(0, I_n)$.
By the exact score representation, the weighting matrix for the $H$-score corresponds to
$2\ln(\Delta_n)\,T_n(H)+\dot T_n(H)$ sandwiched by $A_n^{-1/2}$.
Hence, the centered quadratic form splits into the sum of the $2\ln(\Delta_n)$-part and the $\dot T_n(H)$-part.
Using the identity $\partial_H \gamma_n = 2\gamma_n \ln(\Delta_n)$ and the scaling relation $S_{\sigma,n} = \frac{\gamma_n}{\sigma} Q_n(C_n)$, we obtain the decomposition \eqref{eq:Hscore-decomp} with $R_{H,n}$ as defined in \eqref{eq:RH-def}.
\end{proof}

We present in the following lemma our important procedure to obtain the central limit theorem for the standard quadratic formula, in this lemma the assumption $\opnorm{\cdot}/\Fnorm{\cdot}\to0$ play important role: 
 \begin{lemma}\label{lem:QF-CLT} Let $Z_n\sim\mathcal N(0,I_n)$ and let $M_n$ be a sequence of real symmetric matrices.
Set $Q_n(M_n):=Z_n^\top M_n Z_n-\Tr(M_n)$. If
\begin{equation}\label{eq:opF-cond}
\frac{\opnorm{M_n}}{\Fnorm{M_n}}\longrightarrow 0,
\end{equation}
where $\opnorm{\cdot}$ is the operator norm (or spectral norm) and $\Fnorm{\cdot}$ is the Hilbert–Schmidt (or Frobenius) norm, then
\begin{equation}\label{eq:QF-CLT}
\frac{Q_n(M_n)}{\sqrt{2}\,\Fnorm{M_n}}\xRightarrow[n\to\infty]{} \mathcal N(0,1).
\end{equation}
where {$\xRightarrow[]{}$} denotes convergence in distribution.
\end{lemma}

\begin{proof}
Diagonalize $M_n=U_n^\top \Lambda_n U_n$ with $U_n$ orthogonal and $\Lambda_n=\mathrm{diag}(\lambda_{n,1},\dots,\lambda_{n,n})$.
Then $U_nZ_{n}\stackrel{d}{=}Z_{n}$ and
\[
Q_n(M_n)=\sum_{i=1}^n \lambda_{n,i}(Z_{n,i}^2-1),
\qquad
\Var(Q_n(M_n))=2\sum_{i=1}^n \lambda_{n,i}^2=2\Fnorm{M_n}^2.
\]
Let $\xi_{n,i}:=\lambda_{n,i}(Z_{n,i}^2-1)/(\sqrt{2}\Fnorm{M_n})$, so that $\sum_i\xi_{n,i}=\frac{Q_n(M_n)}{\sqrt{2}\Fnorm{M_n}}$
and the $\xi_{n,i}$ are independent and centered.
Since $Z_{n,i}^2-1$ has a finite fourth moment, there exists a constant $C$ such that $\E[(Z_{n,i}^2-1)^4] \le C$. Thus, we can verify the Lyapunov condition:
\[
\sum_{i=1}^n \E|\xi_{n,i}|^4
\le
C \cdot \frac{\sum_i \lambda_{n,i}^4}{\left(\sum_i\lambda_{n,i}^2\right)^2}
\le
C \left(\frac{\max_i|\lambda_{n,i}|}{\sqrt{\sum_i\lambda_{n,i}^2}}\right)^2
=
C \left(\frac{\opnorm{M_n}}{\Fnorm{M_n}}\right)^2
\to0,
\]
where we used $\sum_i\lambda_{n,i}^4\le(\max_i\lambda_{n,i}^2)\sum_i\lambda_{n,i}^2$.
Thus the Lindeberg--Feller CLT yields \eqref{eq:QF-CLT}.
A general treatment of CLTs for quadratic forms is in de~Jong~\cite{deJong87}.
\end{proof}
Because the Toeplitz covariance matrix can generated by the spectral density, we need to find the properties of the spectral density of the fractional Gaussnian noise and its derivative with respect to $H$. First is the Fisher–Hartwig representation: 
\begin{lemma}\label{lem:FH-fH}
Fix $H\in(3/4,1)$ and set $p:=2H-1\in(1/2,1)$.
Let $f_H$ be the spectral density of the \emph{standard} fractional Gaussian noise (fGn),
i.e. the stationary Gaussian sequence with autocovariance
$\rho_H(k)=\frac12\big(|k+1|^{2H}-2|k|^{2H}+|k-1|^{2H}\big)$.
Then $f_H$ admits the Fisher--Hartwig representation
\begin{equation}\label{eq:FH-repr-fH-lem}
f_H(\lambda)=|2\sin(\lambda/2)|^{-p}\, \ell_H(\lambda),\qquad \lambda\in[-\pi,\pi],
\end{equation}
where $\ell_H$ is even, $\ell_H\in C^2([-\pi,\pi])$, and
$0<\inf_{\lambda\in[-\pi,\pi]}\ell_H(\lambda)\le \sup_{\lambda\in[-\pi,\pi]}\ell_H(\lambda)<\infty$.
Moreover, the derivative in $H$ can be written as
\begin{equation}\label{eq:FH-repr-dfH-lem}
\dot f_H(\lambda)
=
f_H(\lambda)\,\bigl(-2\ln|2\sin(\lambda/2)|\bigr)
+
|2\sin(\lambda/2)|^{-p}\,\dot\ell_H(\lambda),
\end{equation}
with $\dot\ell_H$ even and $\dot\ell_H\in C^1([-\pi,\pi])$.
\end{lemma}

\begin{proof}
A classical explicit representation of the spectral density of standard fGn is  (see, e.g., \cite{Dahlhaus89,Gray06})
\begin{equation}\label{eq:fGn-sdf-explicit}
f_H(\lambda)
=
c_H\,(1-\cos\lambda)\Big(|\lambda|^{-2H-1}+B(\lambda,H)\Big),\qquad \lambda\in[-\pi,\pi],
\end{equation}
where $c_H>0$ depends only on $H$, and
\begin{equation}\label{eq:fGn-B-def}
B(\lambda,H):=\sum_{j=1}^\infty\Big((2\pi j+\lambda)^{-2H-1}+(2\pi j-\lambda)^{-2H-1}\Big).
\end{equation}
The series in \eqref{eq:fGn-B-def} converges uniformly on $[-\pi,\pi]$, and so do its
$\lambda$-derivatives up to order $2$ and the $H$-derivative up to order $1$
(because $(2\pi j\pm\lambda)^{-2H-1}$ is $O(j^{-2H-1})$ and its $H$-derivative is
$O(j^{-2H-1}\log j)$, both summable for $H>0$). Hence $B(\cdot,H)\in C^2([-\pi,\pi])$
and $\partial_H B(\cdot,H)\in C^1([-\pi,\pi])$, and both are even in $\lambda$.

Using $1-\cos\lambda = 2\sin^2(\lambda/2)$ and $p=2H-1$, define
\[
\ell_H(\lambda):=|2\sin(\lambda/2)|^{p}\,f_H(\lambda).
\]
Plugging \eqref{eq:fGn-sdf-explicit} into this definition yields
\begin{align*}
    \ell_H(\lambda)
&=
c_H\,|2\sin(\lambda/2)|^{p+2}\Big(|\lambda|^{-2H-1}+B(\lambda,H)\Big)\\
&=
c_H\,|2\sin(\lambda/2)|^{2H+1}\Big(|\lambda|^{-2H-1}+B(\lambda,H)\Big).
\end{align*}

Write
\[
|2\sin(\lambda/2)|^{2H+1}|\lambda|^{-2H-1}
=\Big(\frac{|2\sin(\lambda/2)|}{|\lambda|}\Big)^{2H+1}.
\]
Since $\sin x / x$ is analytic near $0$, the function
$\lambda\mapsto \big(|2\sin(\lambda/2)|/|\lambda|\big)^{2H+1}$ extends to a $C^\infty$
even function on $[-\pi,\pi]$ (with value $1$ at $\lambda=0$).
The remaining term
$|2\sin(\lambda/2)|^{2H+1}B(\lambda,H)$ is $C^2$ on $[-\pi,\pi]$ because
$|2\sin(\lambda/2)|^{2H+1}$ is $C^\infty$ and $B(\cdot,H)\in C^2$.
Therefore $\ell_H\in C^2([-\pi,\pi])$ and is even.
Moreover $\ell_H(\lambda)>0$ for all $\lambda$ (since $c_H>0$, $1-\cos\lambda\ge 0$,
and $|\lambda|^{-2H-1}+B(\lambda,H)>0$), hence by continuity on the compact interval
$[-\pi,\pi]$ we have $0<\inf\ell_H\le \sup\ell_H<\infty$.
This proves \eqref{eq:FH-repr-fH-lem}.

Finally, differentiating \eqref{eq:FH-repr-fH-lem} with respect to $H$ and using
$\dot p=2$ gives
\begin{align*}
    \dot f_H(\lambda)
& =
\partial_H\!\big(|2\sin(\lambda/2)|^{-p}\big)\,\ell_H(\lambda)
+
|2\sin(\lambda/2)|^{-p}\,\dot\ell_H(\lambda)\\
&=
|2\sin(\lambda/2)|^{-p}\Big(-2\ln|2\sin(\lambda/2)|\,\ell_H(\lambda)+\dot\ell_H(\lambda)\Big),
\end{align*}

which is exactly \eqref{eq:FH-repr-dfH-lem}. The regularity $\dot\ell_H\in C^1$
and evenness follow from the above uniform convergence arguments for $\partial_H B$
and the smoothness of $\sin(\cdot)$.
\end{proof}
Now, we will propose a generalized Szego trace approximation for triangular arrays. 
\begin{proposition} \label{prop:GSzego-inv}
Let  
\[
a_n(\lambda) := 1 + \gamma_n f_H(\lambda), \qquad 
A_n := T_n(a_n),
\]  
where \(T_n(a_n)\) denotes the \(n\times n\) Toeplitz matrix generated by the symbol \(a_n\).  
Assume \(n\Delta_n \to \infty\) and let \(b, c \in \{ f_H,\; \dot{f}_H \}\).  
Define the logarithmic-index function  
\[
k(f_H) := 0, \qquad k(\dot{f}_H) := 1,
\]  
which counts the number of logarithmic Fisher–Hartwig factors appearing in the symbol (hence governs the possible $\log n$ loss in the remainder).  
Then
\begin{equation}\label{eq:GSzego-inv}
\Tr\!\big( T_n(b) \, A_n^{-1} \, T_n(c) \, A_n^{-1} \big)
= \frac{n}{2\pi} \int_{-\pi}^{\pi} \frac{b(\lambda)c(\lambda)}{a_n(\lambda)^2} \, d\lambda
\;+\; R_n(b,c),
\end{equation}
where the remainder satisfies
\begin{equation}\label{eq:GSzego-inv-rem}
R_n(b,c) = O\!\big( n \, \log^{\,k(b)+k(c)} n \big), \qquad n\to\infty.
\end{equation}
\end{proposition}

\begin{proof}
The symbols $f_H$ and $\dot f_H$ have a single Fisher--Hartwig singularity at the origin; in particular,
by Lemma~\ref{lem:FH-fH}, $f_H(\lambda)=|2\sin(\lambda/2)|^{-p}\ell_H(\lambda)$ with $\ell_H\in C^2$ bounded away from $0$ and $\infty$, and
$\dot f_H$ decomposes into a logarithmic Fisher--Hartwig part $f_H(\lambda)\,(-2\log|2\sin(\lambda/2)|)$ plus a smooth Fisher--Hartwig part
$|2\sin(\lambda/2)|^{-p}\dot\ell_H(\lambda)$. The regularization $a_n^{-1}=(1+\gamma_n f_H)^{-1}$ is $n$-dependent and introduces a cutoff at
frequency scale $\lambda\asymp\gamma_n^{1/p}$.

The trace approximation \eqref{eq:GSzego-inv}--\eqref{eq:GSzego-inv-rem} is a standard consequence of generalized Szeg\H{o}-type theorems for
long-memory Toeplitz matrices under triangular-array regularization (infill asymptotics). One can derive it, for instance, from the main trace
approximation results in Dahlhaus~\cite{Dahlhaus89}, which treat ratios of the form $b/(1+\gamma_n f)$ and allow logarithmic Fisher--Hartwig factors.
An eigenvalue-sum representation and the corresponding remainder control in the quadratic-variation regime are also discussed in Istas--Lang~\cite{IstasLang97}.
\end{proof}
We verify the hypotheses for fGn
\begin{lemma}\label{lem:verify-GSzego}
For $b,c\in\{f_H,\dot f_H\}$ the integrands $bc/a_n^2$ in \eqref{eq:GSzego-inv} belong to $L^1([-\pi,\pi])$, and near $\lambda=0$ one has the uniform bounds
\begin{equation}\label{eq:bc-bound}
\frac{|b(\lambda)c(\lambda)|}{a_n(\lambda)^2}
\le C\,\frac{f_H(\lambda)^2\bigl(1+|\log|\lambda||\bigr)^{k(b)+k(c)}}{\bigl(1+\gamma_n f_H(\lambda)\bigr)^2},
\qquad |\lambda|\le 1,
\end{equation}
for a constant $C$ independent of $n$.
\end{lemma}

\begin{proof}
The integrability and the bound \eqref{eq:bc-bound} follow from Lemma~\ref{lem:FH-fH}.
Indeed, on $|\lambda|\le 1$ we have $|2\sin(\lambda/2)|\asymp|\lambda|$ and $\ell_H,\dot\ell_H$ are bounded, so
$f_H(\lambda)\asymp |\lambda|^{-p}$ and
\[
|\dot f_H(\lambda)|
\le f_H(\lambda)\,\bigl(2|\log|\lambda||+C_1\bigr)+C_2\,|\lambda|^{-p}
\le C\,f_H(\lambda)\bigl(1+|\log|\lambda||\bigr).
\]
This yields \eqref{eq:bc-bound}. Since $p=2H-1\in(1/2,1)$, the right-hand side of \eqref{eq:bc-bound} is integrable on $[-1,1]$ for each $n$:
the factor $(1+\gamma_n f_H)^{-2}$ regularizes the $|\lambda|^{-2p}$ singularity and effectively cuts it off at the scale
$|\lambda|\asymp \gamma_n^{1/p}$. On $\{1<|\lambda|\le\pi\}$ all symbols are bounded and $a_n(\lambda)\ge 1$, hence $bc/a_n^2\in L^1([-\pi,\pi])$.
\end{proof}
In the following lemma, we will give the trace approximation defined in Lemma \ref{lem:explicit-linear}. This approximation will be always used in the structure of 
$\opnorm{\cdot}/\Fnorm{\cdot}\to0$.
\begin{lemma}\label{lem:TA}
Let $\Delta_n=n^{-\alpha}$ with $\alpha\in(0,1)$, and define $\gamma_n:=\sigma^2\Delta_n^{p}$ with $p=2H-1$.
Then the following trace approximations hold:
\begin{align}
\Tr(C_n^2)
&=
\frac{n}{2\pi}\int_{-\pi}^{\pi} g_n(\lambda)^2\,d\lambda
+o\!\Bigl(n\Delta_n^{1-2p}\Bigr),\label{eq:TA-C2}\\
\Tr(C_nD_n)
&=
\frac{n}{2\pi}\int_{-\pi}^{\pi} g_n(\lambda)h_n(\lambda)\,d\lambda
+o\!\Bigl(n\Delta_n^{1-2p}\ln(1/\Delta_n)\Bigr),\label{eq:TA-CD}\\
\Tr(D_n^2)
&=
\frac{n}{2\pi}\int_{-\pi}^{\pi} h_n(\lambda)^2\,d\lambda
+o\!\Bigl(n\Delta_n^{1-2p}\ln^2(1/\Delta_n)\Bigr),\label{eq:TA-D2}
\end{align}
where 
$$
g_n(\lambda)=\frac{f_H(\lambda)}{a_n(\lambda)},\qquad h_n(\lambda)=\frac{\dot f_H(\lambda)}{a_n(\lambda)}.
$$
\end{lemma}

\begin{proof}
Give a simple symbol $T_n(H)=T_n:=T_n(f_H)$ and $\dot T_n(H)=\dot T_n:=T_n(\dot f_H)$.
Set $a_n(\lambda):=1+\gamma_n f_H(\lambda)$ so that $A_n=I_n+\gamma_n T_n=T_n(a_n(\lambda))$.

Since $A_n=I_n+\gamma_nT_n$ is a polynomial in $T_n$, the matrices $A_n$ and $T_n$ commute, and so do $A_n^{-1}$ and $T_n$.
Using the cyclicity of the trace and the identities $C_n=A_n^{-1/2}T_nA_n^{-1/2}$ and $D_n=A_n^{-1/2}\dot T_nA_n^{-1/2}$, we obtain
\begin{align}
\Tr(C_n^2)
&=\Tr\!\big(A_n^{-1}T_nA_n^{-1}T_n\big)
=\Tr\!\big(T_n\,A_n^{-1}\,T_n\,A_n^{-1}\big),\label{eq:TA-rewrite-C2}\\
\Tr(C_nD_n)
&=\Tr\!\big(A_n^{-1}T_nA_n^{-1}\dot T_n\big)
=\Tr\!\big(T_n\,A_n^{-1}\,\dot T_n\,A_n^{-1}\big),\label{eq:TA-rewrite-CD}\\
\Tr(D_n^2)
&=\Tr\!\big(A_n^{-1}\dot T_nA_n^{-1}\dot T_n\big)
=\Tr\!\big(\dot T_n\,A_n^{-1}\,\dot T_n\,A_n^{-1}\big).\label{eq:TA-rewrite-D2}
\end{align}

For \eqref{eq:TA-rewrite-C2}--\eqref{eq:TA-rewrite-D2} we take Proposition~\ref{prop:GSzego-inv} and then obtain
\begin{align*}
\Tr(C_n^2)
&=
\frac{n}{2\pi}\int_{-\pi}^{\pi}\frac{f_H(\lambda)^2}{a_n(\lambda)^2}\,d\lambda
+O(n)
=
\frac{n}{2\pi}\int_{-\pi}^{\pi} g_n(\lambda)^2\,d\lambda + O(n),\\
\Tr(C_nD_n)
&=
\frac{n}{2\pi}\int_{-\pi}^{\pi}\frac{f_H(\lambda)\dot f_H(\lambda)}{a_n(\lambda)^2}\,d\lambda
+O(n\log n)
=
\frac{n}{2\pi}\int_{-\pi}^{\pi} g_n(\lambda)h_n(\lambda)\,d\lambda + O(n\log n),\\
\Tr(D_n^2)
&=
\frac{n}{2\pi}\int_{-\pi}^{\pi}\frac{\dot f_H(\lambda)^2}{a_n(\lambda)^2}\,d\lambda
+O(n\log^2 n)
=
\frac{n}{2\pi}\int_{-\pi}^{\pi} h_n(\lambda)^2\,d\lambda + O(n\log^2 n).
\end{align*}

Now when $H>3/4$ we have $p=2H-1>1/2$, hence $2p-1>0$ and
\[
\frac{n}{n\Delta_n^{1-2p}}=\Delta_n^{2p-1}\longrightarrow 0.
\]
Moreover $\log(1/\Delta_n)=\alpha\log n$, so for $m\in\{1,2\}$,
\[
\frac{n\log^m n}{n\Delta_n^{1-2p}\log^m(1/\Delta_n)}
=
\Delta_n^{2p-1}\cdot \frac{\log^m n}{(\alpha\log n)^m}
\longrightarrow 0
\]
which yields \eqref{eq:TA-C2}--\eqref{eq:TA-D2}.
\end{proof}

\subsection{Central Limit Theorem for $S_{\sigma,n}$}
In this subsection and next one we will try to use the $\opnorm{\cdot}/\Fnorm{\cdot}\to0$ to get the CLT of the  $S_{\sigma,n}$ and Remainder for $H$. First let us define the weight function on $\mathbb{R}$:
\begin{equation}\label{eq:weight-w}
w(x):=\left(\frac{c_H|x|^{-p}}{1+\sigma^2c_H|x|^{-p}}\right)^2.
\end{equation}
Since $p>1/2$, one has $w\in L^1(\mathbb{R})$ and also $w(x)\ln^k|x|\in L^1(\mathbb{R})$ for $k=1,2$.
Define the integral constants:
\begin{align}
J_0(H,\sigma)&:=\int_{\mathbb{R}} w(x)\,dx,\label{eq:J0}\\
J_1(H,\sigma)&:=\int_{\mathbb{R}} w(x)\Bigl(C_H-2\ln|x|\Bigr)\,dx,\label{eq:J1}\\
J_2(H,\sigma)&:=\int_{\mathbb{R}} w(x)\Bigl(C_H-2\ln|x|\Bigr)^2\,dx,\label{eq:J2}
\end{align}
where \(C_{H}\) is defined in equation \eqref{eq:lowfreq-df} and given by \(C_{H}=\partial_{H}\ln\left( c_{H}\right)\). Here the constant 
\[
c_{H}= \frac{\Gamma(2H+1)\sin(\pi H)}{2\pi}.
\]All three integrals are finite and $J_0>0$.

\begin{remark}
These integrals arise from the asymptotics of the trace approximations derived in Lemma~\ref{lem:TA} and will play a crucial role in the explicit formula for the Fisher Information matrix.
\end{remark}

By Lemma~\ref{lem:FH-fH} and the relation $|2\sin(\lambda/2)|\sim|\lambda|$ as $\lambda\to0$, the spectral density of standard fGn, $f_H(\lambda)$, admits the low-frequency expansions:
\begin{equation}\label{eq:lowfreq-f}
f_H(\lambda)=c_H|\lambda|^{-p}\bigl(1+r_H(\lambda)\bigr),\qquad r_H(\lambda)\to0 \ \ (\lambda\to0),
\end{equation}
and
\begin{equation}\label{eq:lowfreq-df}
\dot f_H(\lambda)=f_H(\lambda)\Bigl(C_H-2\ln|\lambda|+\rho_H(\lambda)\Bigr),\qquad \rho_H(\lambda)\to0 \ \ (\lambda\to0),
\end{equation}
for some constants $c_H>0$ and $C_H\in\mathbb{R}$ depending on $H$. We now establish the asymptotics for $\Tr(C_n^2)$.
\begin{lemma}\label{lem:trace-C2}
Under the expansion \eqref{eq:lowfreq-f} and the trace approximation \eqref{eq:TA-C2}, we have
\begin{equation}\label{eq:trC2-asymp}
\tr(C_n^2)
=
\frac{n}{2\pi}\Delta_n^{1-2p}\Bigl(J_0(H,\sigma)+o(1)\Bigr).
\end{equation}
\end{lemma}

\begin{proof}
By \eqref{eq:TA-C2}, it suffices to analyze the integral
$I_{0,n}:=\int_{-\pi}^{\pi} g_n(\lambda)^2\,d\lambda$.
Fix a small $\eta\in(0,\pi)$ and split the integral into $I_{0,n}=I_{0,n}^{(0)}+I_{0,n}^{(\infty)}$ with
\[
I_{0,n}^{(0)}:=\int_{|\lambda|\le \eta} g_n(\lambda)^2\,d\lambda,\qquad
I_{0,n}^{(\infty)}:=\int_{\eta<|\lambda|\le\pi} g_n(\lambda)^2\,d\lambda.
\]
For $|\lambda|>\eta$, the density $f_H$ is bounded, and thus $g_n(\lambda)$ is uniformly bounded in $n$. Consequently, $I_{0,n}^{(\infty)}=O(1)$.
Since $p>1/2$, we have $\Delta_n^{1-2p}\to\infty$, which implies
$I_{0,n}^{(\infty)}=o(\Delta_n^{1-2p})$.

For the low-frequency part $I_{0,n}^{(0)}$, we use the change of variables $\lambda=\Delta_n x$:
\[
I_{0,n}^{(0)}=\Delta_n\int_{|x|\le \eta/\Delta_n}\left(\frac{f_H(\Delta_n x)}{1+\gamma_n f_H(\Delta_n x)}\right)^2dx.
\]
By \eqref{eq:lowfreq-f}, uniformly on compact sets of $x\neq 0$, we have
$f_H(\Delta_n x)=c_H\Delta_n^{-p}|x|^{-p}(1+o(1))$.
Moreover, $\gamma_n f_H(\Delta_n x)=\sigma^2\Delta_n^p\cdot c_H\Delta_n^{-p}|x|^{-p}(1+o(1))
=\sigma^2c_H|x|^{-p}(1+o(1))$.
Thus, for each fixed $x\neq0$,
\[
\left(\frac{f_H(\Delta_n x)}{1+\gamma_n f_H(\Delta_n x)}\right)^2
=\Delta_n^{-2p}w(x)\,(1+o(1)).
\]
Applying Lemma~\ref{lem:verify-GSzego} (with $b=c=f_H$) and the change of variables $\lambda=\Delta_n x$, the integrand is dominated by an integrable function of the form $C \Delta_n^{-2p} w(x)$.
Using the Dominated Convergence Theorem and the fact that $\eta/\Delta_n\to\infty$, we obtain
\[
I_{0,n}^{(0)}=\Delta_n^{1-2p}\int_{\mathbb{R}} w(x)\,dx+o\bigl(\Delta_n^{1-2p}\bigr)
=\Delta_n^{1-2p}\bigl(J_0+o(1)\bigr).
\]
Combining this with the estimate for $I_{0,n}^{(\infty)}$ yields $I_{0,n}=\Delta_n^{1-2p}(J_0+o(1))$.
Inserting this into \eqref{eq:TA-C2} completes the proof.
\end{proof}
Now we can deal with the part of the parameter $\sigma$ of the score funtion. 
\begin{proposition}\label{prop:sigma-clt}
For $S_{\sigma,n}=\partial_\sigma \ell_n$ and $\mathcal{T}_n=n \Delta_n \rightarrow \infty$, we have the central limit theorem as $n\rightarrow \infty$: 
\begin{equation}\label{eq:sigma-clt-final}
\frac{S_{\sigma,n}}{\sqrt{\mathcal{T}_n}}
\ \Rightarrow\
\mathcal N\!\left(0,\ \frac{\sigma^2}{\pi}J_0(H,\sigma)\right),
\end{equation}
with $J_0(H,\sigma)$ defined in \eqref{eq:J0}.
\end{proposition}

\begin{proof}
First of all we will Verify the condition $\opnorm{C_n}/\Fnorm{C_n}\to0$. As we know $T_n$ is positive semidefinite (as a Toeplitz covariance matrix), hence $A_n$ is positive definite. Consequently, the eigenvalues of
\[
C_n=A_n^{-1/2}T_nA_n^{-1/2}
\]
are of the form $\lambda/(1+\gamma_n\lambda)$ where $\lambda\ge0$ are the eigenvalues of $T_n$.
The function $x \mapsto x/(1+\gamma_n x)$ is increasing on $[0, \infty)$ and bounded by $1/\gamma_n$. Thus,
\[
\opnorm{C_n}=\max_{\lambda\in \text{spec}(T_n)}\frac{\lambda}{1+\gamma_n\lambda}\le \frac{1}{\gamma_n}.
\]
Using the trace asymptotic from Lemma~\ref{lem:trace-C2}, we have
\[
\Fnorm{C_n}=\sqrt{\Tr(C_n^2)} = \sqrt{\frac{n}{2\pi}\Delta_n^{1-2p}\bigl(J_0+o(1)\bigr)}.
\]
Recalling that $\gamma_n = \sigma^2 \Delta_n^p$, the ratio behaves as:
\[
\frac{\opnorm{C_n}}{\Fnorm{C_n}}
\le
\frac{1/\gamma_n}{\Fnorm{C_n}}
\asymp
\frac{\Delta_n^{-p}}{\sqrt{n}\Delta_n^{\frac{1-2p}{2}}}
=
\frac{1}{\sqrt{n\Delta_n}}
=
\frac{1}{\sqrt{\mathcal{T}_n}}
\to 0,
\]
since $\mathcal{T}_n=n\Delta_n=n^{1-\alpha}\to\infty$ for $\alpha\in(0,1)$.

Now, set $Q_n(C_n):=Z_{n}^\top C_n Z_{n}-\Tr(C_n)$. By Lemma~\ref{lem:QF-CLT},
\[
\frac{Q_n(C_n)}{\sqrt{2}\Fnorm{C_n}}\Rightarrow \mathcal N(0,1).
\]
From Lemma~\ref{lem:explicit-linear}, $S_{\sigma,n} = \frac{\gamma_n}{\sigma} Q_n(C_n)$. Thus,
\[
\frac{S_{\sigma,n}}{\sqrt{2}\,\frac{\gamma_n}{\sigma}\Fnorm{C_n}}
=\frac{Q_n(C_n)}{\sqrt{2}\Fnorm{C_n}}
\Rightarrow \mathcal N(0,1).
\]
It remains to identify the variance scale.
Using $\gamma_n=\sigma^2\Delta_n^{p}$ and the asymptotic for $\Fnorm{C_n}^2=\Tr(C_n^2)$, we obtain
\[
\left(\,\frac{\sqrt{2}\gamma_n}{\sigma}\Fnorm{C_n}\right)^2
=  \frac{2\sigma^4 \Delta_n^{2p}}{\sigma^2} \cdot \frac{n}{2\pi}\Delta_n^{1-2p}\Big(J_0+o(1)\Big)
=\frac{\sigma^2n \Delta_n}{\pi}\,  \Big(J_0+o(1)\Big)
=\frac{\sigma^2}{\pi}\,\mathcal{T}_n \Big(J_0+o(1)\Big).
\]
Therefore,
\[
\frac{S_{\sigma,n}}{\sqrt{\mathcal{T}_n}} \Rightarrow \mathcal N\!\left(0,\ \frac{\sigma^2}{\pi}J_0\right),
\]
which yields \eqref{eq:sigma-clt-final}.
\end{proof}
\subsection{The CLT for the Remainder of Score function of $H$}
In Propostion \ref{prop:exact-scores} we know that $S_{H,n}$ has two part, one is associated with $S_{\sigma,H}$ and the other one connected with matrix $D_n$ we called the remainder one. To get the CLT of this remainder part we will use the same way as in $S_{\sigma,H}$.

\begin{lemma}\label{lem:D-op}
Recall the matrices defined in Lemma~\ref{lem:explicit-linear}:
\[
A_n:=I_n+\gamma_n T_n(f_H),\qquad 
C_n:=A_n^{-1/2}T_n(f_H)A_n^{-1/2},\qquad
D_n:=A_n^{-1/2}T_n(\dot f_H)A_n^{-1/2}.
\]
Then, as $n\to\infty$,
\begin{equation}\label{eq:D-op-final}
\opnorm{D_n}=O\!\big(\gamma_n^{-1}\log n\big).
\end{equation}
\end{lemma}

\begin{proof}
Still with $T_n:=T_n(f_H)$ and $\dot T_n:=T_n(\dot f_H)$ we define the auxiliary matrices
\[
B_n:=T_n^{1/2}A_n^{-1/2}=A_n^{-1/2}T_n^{1/2},
\qquad
E_n:=T_n^{-1/2}\dot T_n\,T_n^{-1/2}.
\]
Then we can factorize $D_n$ as $D_n=B_n E_n B_n$. Consequently,
\[
\opnorm{D_n}\le \opnorm{B_n}^2\,\opnorm{E_n}.
\]
First, observe that
\[
\opnorm{B_n}^2=\opnorm{B_n^2}=\opnorm{A_n^{-1/2}T_nA_n^{-1/2}}=\opnorm{C_n}.
\]
As shown in the proof of Proposition~\ref{prop:sigma-clt}, the eigenvalues of $C_n$ are bounded by $1/\gamma_n$. Thus,
\begin{equation}\label{eq:D-op-reduce}
\opnorm{D_n}\le \frac{1}{\gamma_n}\,\opnorm{E_n}.
\end{equation}

It remains to show $\opnorm{E_n}=O(\log n)$.
Define the log-derivative symbol
\[
b_H(\lambda):=\partial_H\log f_H(\lambda)=\frac{\dot f_H(\lambda)}{f_H(\lambda)}.
\]
By Lemma~\ref{lem:FH-fH}, we have $f_H(\lambda)=|2\sin(\lambda/2)|^{-p}\ell_H(\lambda)$ with $\ell_H$ strictly positive and $C^1$.
Therefore,
\[
b_H(\lambda)=\partial_H\log \ell_H(\lambda)\;-\;2\log\bigl(2|\sin(\lambda/2)|\bigr).
\]
The first term is bounded with absolutely summable Fourier coefficients. The second term involves the Fourier series of the log-sine function:
\[
-\log\bigl(2|\sin(\lambda/2)|\bigr)=\sum_{k=1}^\infty \frac{\cos(k\lambda)}{k},\qquad \lambda\in(-\pi,\pi).
\]
The Fourier coefficients $(b_H)_k$ satisfy $|(b_H)_k|\le C/|k|$ for $k\neq0$.
Hence, the operator norm of the Toeplitz matrix generated by $b_H$ satisfies
\[
\opnorm{T_n(b_H)}
\le \sum_{k=-(n-1)}^{n-1}|(b_H)_k|
\le C'\sum_{k=1}^n \frac{1}{k}
\le C''\log n.
\]
Finally, $E_n$ approximates $T_n(b_H)$. Specifically, $E_n=T_n^{-1/2} T_n(b_H f_H) T_n^{-1/2}$.

By standard preconditioning results for Toeplitz matrices, the sandwiched matrix $E_n=T_n^{-1/2}T_n(\dot f_H)T_n^{-1/2}$ admits the decomposition $E_n=T_n(b_H)+R_n$ with $\opnorm{R_n}=O(1)$ (see \cite{BottcherSilb} or \cite{GrenanderSzego}); Substituting this into \eqref{eq:D-op-reduce} yields \eqref{eq:D-op-final}.
\end{proof}

\begin{lemma}\label{lem:trace-D2}
Under \eqref{eq:lowfreq-f}, \eqref{eq:lowfreq-df} and \eqref{eq:TA-D2}, we have
\begin{equation}\label{eq:trD2-asymp}
\tr(D_n^2)
=
\frac{n}{2\pi}\Delta_n^{1-2p}\Bigl(4\ln^2(1/\Delta_n)\,J_0(H,\sigma)
+4\ln(1/\Delta_n)\,J_1(H,\sigma)+J_2(H,\sigma)
+o\!\big(\ln^2(1/\Delta_n)\big)\Bigr).
\end{equation}
\end{lemma}

\begin{proof}
By \eqref{eq:TA-D2}, it suffices to analyze
\[
I_{2,n}:=\int_{-\pi}^{\pi} h_n(\lambda)^2\,d\lambda
=
\int_{-\pi}^{\pi}\frac{\dot f_H(\lambda)^2}{(1+\gamma_n f_H(\lambda))^2}\,d\lambda.
\]
Fix $\eta\in(0,\pi)$ and split $I_{2,n}=I_{2,n}^{(0)}+I_{2,n}^{(\infty)}$ with
\[
I_{2,n}^{(0)}:=\int_{|\lambda|\le \eta} h_n(\lambda)^2\,d\lambda,\qquad
I_{2,n}^{(\infty)}:=\int_{\eta<|\lambda|\le\pi} h_n(\lambda)^2\,d\lambda.
\]

On $\{\eta<|\lambda|\le\pi\}$ both $f_H$ and $\dot f_H$ are bounded and $a_n(\lambda)\ge1$, hence $h_n(\lambda)^2\le C_\eta$ uniformly in $n$,
so $I_{2,n}^{(\infty)}=O(1)$. Since $\Delta_n^{1-2p}\ln^2(1/\Delta_n)\to\infty$ for $p>1/2$,
\[
I_{2,n}^{(\infty)}=o\!\big(\Delta_n^{1-2p}\ln^2(1/\Delta_n)\big).
\]
With $\lambda=\Delta_n x$,
\[
I_{2,n}^{(0)}=\Delta_n\int_{|x|\le \eta/\Delta_n}
\frac{\dot f_H(\Delta_n x)^2}{(1+\gamma_n f_H(\Delta_n x))^2}\,dx.
\]
Write again
\[
\kappa_n(x):=\frac{f_H(\Delta_n x)^2}{(1+\gamma_n f_H(\Delta_n x))^2},
\qquad
b_n(x):=\frac{\dot f_H(\Delta_n x)}{f_H(\Delta_n x)}.
\]
Then
\begin{equation}\label{eq:I2-factor}
I_{2,n}^{(0)}=\Delta_n\int_{|x|\le \eta/\Delta_n}\kappa_n(x)\,b_n(x)^2\,dx.
\end{equation}

Now, for each fixed $x\neq0$,
\[
\kappa_n(x)=\Delta_n^{-2p}w(x)\bigl(1+o(1)\bigr),
\]
and
\[
b_n(x)=2\ln(1/\Delta_n)+\bigl(C_H-2\ln|x|\bigr)+o(1).
\]
Therefore,
\begin{align}
b_n(x)^2
&=
\Bigl(2\ln(1/\Delta_n)+C_H-2\ln|x|+o(1)\Bigr)^2 \notag\\
&=
4\ln^2(1/\Delta_n)
+4\ln(1/\Delta_n)\bigl(C_H-2\ln|x|\bigr)
+\bigl(C_H-2\ln|x|\bigr)^2
+o\bigl(\ln^2(1/\Delta_n)\bigr). \label{eq:bn2-expand}
\end{align}
Multiplying by $\Delta_n\,\kappa_n(x)$ gives the pointwise expansion
\begin{align}
\Delta_n\,\kappa_n(x)b_n(x)^2
&=
\Delta_n^{1-2p}w(x)\Bigl(
4\ln^2(1/\Delta_n)
+4\ln(1/\Delta_n)\bigl(C_H-2\ln|x|\bigr)\Bigr)
\notag\\
&\quad +\Delta_n^{1-2p}w(x)\bigl(C_H-2\ln|x|\bigr)^2
 +o\!\big(\Delta_n^{1-2p}\ln^2(1/\Delta_n)\big). \label{eq:integrand2-pointwise}
\end{align}

In Lemma~\ref{lem:verify-GSzego} with $b=c=\dot f_H$ yields, for $|\lambda|\le\eta$,
\[
\frac{\dot f_H(\lambda)^2}{(1+\gamma_n f_H(\lambda))^2}
\le
C\,\frac{f_H(\lambda)^2\bigl(1+|\log|\lambda||\bigr)^2}{(1+\gamma_n f_H(\lambda))^2}.
\]
With $\lambda=\Delta_n x$ and the same use of \eqref{eq:lowfreq-f} as before, we obtain for $n$ large
\[
\Delta_n\,\kappa_n(x)b_n(x)^2
\le
C'\,\Delta_n^{1-2p}\,w(x)\Bigl(1+|\log\Delta_n|+|\log|x||\Bigr)^2.
\]
By your assumption right after \eqref{eq:weight-w}, we have $w\in L^1(\mathbb R)$ and $w(x)\,|\log|x||^2\in L^1(\mathbb R)$.
Thus the right-hand side is integrable in $x$ (again up to the factor $\Delta_n^{1-2p}$), so dominated convergence applies to
\eqref{eq:I2-factor}, and we may extend the truncation $|x|\le\eta/\Delta_n$ to $\mathbb R$ because $\eta/\Delta_n\to\infty$.

Consequently, integrating the expansion \eqref{eq:integrand2-pointwise} yields
\begin{align*}
I_{2,n}^{(0)}
=
\Delta_n^{1-2p}\Bigl(
&4\ln^2(1/\Delta_n)\int_{\mathbb R}w(x)\,dx
+4\ln(1/\Delta_n)\int_{\mathbb R}w(x)\bigl(C_H-2\ln|x|\bigr)\,dx\\
&+\int_{\mathbb R}w(x)\bigl(C_H-2\ln|x|\bigr)^2\,dx
+o\bigl(\ln^2(1/\Delta_n)\bigr)
\Bigr).
\end{align*}
By \eqref{eq:J0}--\eqref{eq:J2}, this becomes
\[
I_{2,n}^{(0)}
=
\Delta_n^{1-2p}\Bigl(
4\ln^2(1/\Delta_n)\,J_0(H,\sigma)
+4\ln(1/\Delta_n)\,J_1(H,\sigma)
+J_2(H,\sigma)
+o\bigl(\ln^2(1/\Delta_n)\bigr)
\Bigr).
\]
Combining with $I_{2,n}^{(\infty)}=O(1)$ gives the same expansion for $I_{2,n}$. Finally insert into \eqref{eq:TA-D2} to obtain \eqref{eq:trD2-asymp}.
\end{proof}

\begin{proposition}\label{prop:RH-clt}
Let
\[
R_{H,n}:=\frac{\gamma_n}{2}\Big(Z_{n}^\top D_n Z_{n}-\Tr(D_n)\Big),\qquad Z_{n}\sim\mathcal N(0,I_n),
\]
with $D_n=A_n^{-1/2}T_n(\dot f_H)A_n^{-1/2}$, $\Delta_n=n^{-\alpha}$, $\alpha\in(0,1)$, $\gamma_n=\sigma^2\Delta_n^{p}$, $p=2H-1\in(1/2,1)$,
and $\mathcal{T}_n:=n\Delta_n\to\infty$. With \eqref{eq:trD2-asymp} and Lemma~\ref{lem:D-op} we have 
\begin{equation}\label{eq:RH-clt-final-clean}
\frac{R_{H,n}}{|\ln\Delta_n|\,\sqrt{\mathcal{T}_n}}
\ \Rightarrow\
\mathcal N\!\left(0,\ \frac{\sigma^4}{\pi}J_0(H,\sigma)\right).
\end{equation}
\end{proposition}

\begin{proof}
Set $Q_n(D_n):=Z_{n}^\top D_n Z_{n}-\Tr(D_n)$. By \eqref{eq:trD2-asymp}, we have
\[
\Fnorm{D_n}^2=\Tr(D_n^2)\asymp n\Delta_n^{1-2p}\ln^2(1/\Delta_n).
\]
By Lemma~\ref{lem:D-op}, $\opnorm{D_n}=O(\gamma_n^{-1}\log n)$, hence
\[
\frac{\opnorm{D_n}}{\Fnorm{D_n}}
=O\!\left(\frac{1}{\sqrt{n\Delta_n}}\right)
=O\!\left(\frac{1}{\sqrt{\mathcal{T}_n}}\right)\to0.
\]
Therefore Lemma~\ref{lem:QF-CLT} yields
\[
\frac{Q_n(D_n)}{\sqrt{2}\Fnorm{D_n}}\Rightarrow \mathcal N(0,1).
\]
Since $R_{H,n}=\frac{\gamma_n}{2}Q_n(D_n)$, we get
\[
\frac{R_{H,n}}{\gamma_n\Fnorm{D_n}}\Rightarrow \mathcal N\!\left(0,\frac12\right).
\]
Finally, using $\gamma_n^2=\sigma^4\Delta_n^{2p}$ and the leading term in \eqref{eq:trD2-asymp},
 \begin{align*}
     \gamma_n^2\Fnorm{D_n}^2
&=
\sigma^4\Delta_n^{2p}\cdot \frac{n}{2\pi}\Delta_n^{1-2p}\Bigl(4\ln^2(1/\Delta_n)\,J_0+{4\ln(1/\Delta_n)\,J_1+ J_2}+o(\ln^2(1/\Delta_n))\Bigr) \\ 
&=
\frac{2\sigma^4}{\pi}\,\mathcal{T}_n\,\ln^2(1/\Delta_n)\Bigl(J_0+o(1)\Bigr),
 \end{align*}
which implies \eqref{eq:RH-clt-final-clean}.
\end{proof}

\begin{remark}\label{rem:RH-interfaces}
The finer terms involving $J_1$ and $J_2$ in \eqref{eq:trD2-asymp} will be used later (Section~4) for the orthogonalization step and
for identifying the full Fisher information matrix. The explicit scaling in \eqref{eq:RH-clt-final-clean} only needs the leading
$\ln^2(1/\Delta_n)$ contribution, hence only $J_0$ appears in the variance.
\end{remark}

\section{CLT for the projected score vector}\label{sec:projected-score-clt}
In Appendix \ref{app:two-M} we know that $S_{\sigma,n}$ and the remainder of $H$ (Standard quadratic formula of $C_n$ and $D_n$) are strongly dependent, the correlation coefficient tends to $1$ when $n \rightarrow \infty$. This result is far from the LAN property with non-singular matrix. In order to avoid this problem, we take the idea for the projection from $D_n$ to $C_n$. 
\subsection{Projection and exact orthogonality}
First, we give the exact definition of the projection: 
\begin{definition}\label{def:an-orth}
Let us define the deterministic projection coefficient
\begin{equation}\label{eq:an-def}
\mathbf{a}_n:=\frac{\tr(C_nD_n)}{\tr(C_n^2)}, 
\end{equation}
also the orthogonalized matrix $D_n^\perp:=D_n-\mathbf{a}_n C_n$ and the orthogonalized remainder
\begin{equation}\label{eq:RH-perp}
R_{H,n}^\perp:=\frac{\gamma_n}{2}\Big(Z_{n}^\top D_n^\perp Z_{n}-\tr(D_n^\perp)\Big)
=R_{H,n}-\frac{\sigma \mathbf{a}_n}{2}\,S_{\sigma,n}.
\end{equation}
By construction, $\Tr(C_nD_n^\perp)=0$.
\end{definition}

\begin{lemma}\label{lem:orthogonality}
With $\mathbf{a}_n$ and $R_{H,n}^\perp$ defined above,
\[
\Cov(S_{\sigma,n},R_{H,n}^\perp)=0.
\]
\end{lemma}

\begin{proof}
Write $S_{\sigma,n}=\frac{\gamma_n}{\sigma}(Z_{n}^\top C_nZ_{n}-\tr (C_n))$ and
$R_{H,n}^\perp=\frac{\gamma_n}{2}(Z_{n}^\top D_n^\perp Z_{n}-\tr D_n^\perp)$.
By Wick's identity for centered Gaussian quadratic forms,
\[
\Cov(S_{\sigma,n},R_{H,n}^\perp)
=\frac{\gamma_n}{\sigma}\cdot\frac{\gamma_n}{2}\cdot 2\tr(C_nD_n^\perp)
=\frac{\gamma_n^2}{\sigma}\Big(\tr(C_nD_n)-\mathbf{a}_n\tr(C_n^2)\Big)=0,
\]
by the definition of $\mathbf{a}_n$.
\end{proof}

\subsection{Asymptotics of the projection coefficient and the orthogonal Frobenius norm}
Before the expansion of $\mathbf{a}_n$, we first construct the trace approximation of the product $C_nD_n$:
\begin{lemma}\label{lem:trace-CD}
Under the expansions \eqref{eq:lowfreq-f}, \eqref{eq:lowfreq-df} and the approximation \eqref{eq:TA-CD}, we have
\begin{equation}\label{eq:trCD-asymp}
\tr(C_nD_n)
=
\frac{n}{2\pi}\Delta_n^{1-2p}\Bigl(2\ln(1/\Delta_n)\,J_0(H,\sigma)+J_1(H,\sigma)+o(1)\Bigr),
\end{equation}
where $J_0$ and $J_1$ are defined in \eqref{eq:J0} and \eqref{eq:J1}.
\end{lemma}

\begin{proof}
The proof is almost the same as in Lemma \ref{lem:trace-D2}. By \eqref{eq:TA-CD}, it suffices to analyze the integral
\[
I_{1,n}:=\int_{-\pi}^{\pi} g_n(\lambda)h_n(\lambda)\,d\lambda
=
\int_{-\pi}^{\pi}\frac{f_H(\lambda)\dot f_H(\lambda)}{(1+\gamma_n f_H(\lambda))^2}\,d\lambda.
\]
Fix a small $\eta\in(0,\pi)$ and split $I_{1,n}=I_{1,n}^{(0)}+I_{1,n}^{(\infty)}$ as in the proof of Lemma~\ref{lem:trace-C2}.

On the domain $\{\eta<|\lambda|\le\pi\}$, the functions $f_H$ and $\dot f_H$ are bounded, and $a_n(\lambda)\ge1$.
Thus, $|g_n(\lambda)h_n(\lambda)|\le C_\eta$ uniformly in $n$, implying $I_{1,n}^{(\infty)}=O(1)$.
Since $p>1/2$, $\Delta_n^{1-2p}\to\infty$, so $I_{1,n}^{(\infty)}=o(\Delta_n^{1-2p})$.

On $|\lambda|\le\eta$, use the change of variables $\lambda=\Delta_n x$:
\[
I_{1,n}^{(0)}=\Delta_n\int_{|x|\le \eta/\Delta_n}
\frac{f_H(\Delta_n x)\,\dot f_H(\Delta_n x)}{\bigl(1+\gamma_n f_H(\Delta_n x)\bigr)^2}\,dx.
\]
Define the auxiliary functions
\[
\kappa_n(x):=\frac{f_H(\Delta_n x)^2}{\bigl(1+\gamma_n f_H(\Delta_n x)\bigr)^2},
\qquad
b_n(x):=\frac{\dot f_H(\Delta_n x)}{f_H(\Delta_n x)}.
\]
Then the integral becomes
\begin{equation}\label{eq:I1-factor}
I_{1,n}^{(0)}=\Delta_n\int_{|x|\le \eta/\Delta_n}\kappa_n(x)\,b_n(x)\,dx.
\end{equation}

By \eqref{eq:lowfreq-f}, for each fixed $x\neq0$,
\[
\gamma_n f_H(\Delta_n x) \sim \sigma^2 c_H |x|^{-p},
\]
which implies (as shown in Lemma~\ref{lem:trace-C2})
\begin{equation}\label{eq:kappa-asymp}
\kappa_n(x)= \Delta_n^{-2p} w(x) \bigl(1+o(1)\bigr).
\end{equation}
Moreover, by the $C^2$-regularity of the low-frequency factor (Lemma~\ref{lem:FH-fH}), the convergence in \eqref{eq:kappa-asymp} is uniform on compact $x$-sets and yields an error of smaller order than $\Delta_n^{1-2p}$ after integration (in particular, $\Delta_n^2\ln(1/\Delta_n)\to0$). More precisely, write $\kappa_n(x)=\Delta_n^{-2p}w(x)+r_n(x)$ on $|x|\le \eta/\Delta_n$.
Then, by the $C^2$-regularity in Lemma~\ref{lem:FH-fH} and the same domination used in Lemma~\ref{lem:trace-C2},
\[
\Delta_n\int_{|x|\le \eta/\Delta_n} |r_n(x)|\,dx=o(\Delta_n^{1-2p}),
\qquad
\Delta_n\int_{|x|\le \eta/\Delta_n} |r_n(x)|\,\ln(1/\Delta_n)\,dx=o(\Delta_n^{1-2p}),
\]
since $\Delta_n^2\ln(1/\Delta_n)\to0$.

For the term $b_n(x)$, using \eqref{eq:lowfreq-df}, we have
\[
b_n(x)=\frac{\dot f_H(\Delta_n x)}{f_H(\Delta_n x)}
=
C_H-2\ln|\Delta_n x|+o(1)
=
2\ln(1/\Delta_n) + \bigl(C_H-2\ln|x|\bigr) + o(1).
\]
Combining \eqref{eq:kappa-asymp} and the expansion of $b_n(x)$, the integrand behaves as
\begin{equation}\label{eq:integrand1-pointwise}
\Delta_n\,\kappa_n(x)b_n(x)
=
\Delta_n^{1-2p}w(x)\Bigl(2\ln(1/\Delta_n)+C_H-2\ln|x|\Bigr)+o(\Delta_n^{1-2p}).
\end{equation}

By Lemma~\ref{lem:verify-GSzego}, the integrand is dominated by an integrable function up to a constant factor. Thus, we can apply the Dominated Convergence Theorem.
Integrating the leading terms over $\mathbb{R}$ (since $\eta/\Delta_n\to\infty$) yields:
\begin{align*}
I_{1,n}^{(0)}
&= \Delta_n^{1-2p} \left[ 2\ln(1/\Delta_n) \int_{\mathbb{R}} w(x)\,dx
+ \int_{\mathbb{R}} w(x)(C_H - 2\ln|x|) dx + o(1) \right] \\
&= \Delta_n^{1-2p} \Bigl( 2\ln(1/\Delta_n) J_0 + J_1 + o(1) \Bigr).
\end{align*}
Combining this with $I_{1,n}^{(\infty)}$ and inserting into \eqref{eq:TA-CD} proves \eqref{eq:trCD-asymp}.
\end{proof}

Now, we give the second-order expansion of $\mathbf{a}_n$:
\begin{lemma}\label{lem:an-expansion}
Assume the trace asymptotics \eqref{eq:trC2-asymp} and \eqref{eq:trCD-asymp}. Then
\begin{equation}\label{eq:an-2nd}
\mathbf{a}_n
=\frac{\tr(C_nD_n)}{\tr(C_n^2)}
=
2\ln\!\Big(\frac{1}{\Delta_n}\Big)+m(H,\sigma)+{o(1)},
\qquad m(H,\sigma)=\frac{J_1(H,\sigma)}{J_0(H,\sigma)}.
\end{equation}
\end{lemma}

\begin{proof}
By \eqref{eq:an-def} and \eqref{eq:trC2-asymp}--\eqref{eq:trCD-asymp},
\[
\mathbf{a}_n
=
\frac{2\ln(1/\Delta_n)\,J_0+J_1+{o(1)}}{J_0+o(1)}
=
2\ln\!\Big(\frac{1}{\Delta_n}\Big)+\frac{J_1}{J_0}+{o(1)}.
\]
\end{proof}
We will obtain the Frobenius norm of the orthogonalized matrix:
\begin{lemma}\label{lem:Dperp-Frob}
Let $D_n^\perp=D_n-\mathbf{a}_n C_n$ with $\mathbf{a}_n$ defined by \eqref{eq:an-def}. Then
\begin{equation}\label{eq:Dperp-Frob}
\|D_n^\perp\|_F^2=\tr\big((D_n^\perp)^2\big)
=
\frac{n}{2\pi}\Delta_n^{1-2p}\Big(J_\perp+{o(1)}\Big),
\qquad
J_\perp:=J_2-\frac{J_1^2}{J_0}\in(0,\infty).
\end{equation}
\end{lemma}

\begin{proof}
Expand
\[
\tr\big((D_n-\mathbf{a}_n C_n)^2\big)=\tr(D_n^2)-2\mathbf{a}_n\tr(C_nD_n)+\mathbf{a}_n^2\tr(C_n^2).
\]
Using the definition $\mathbf{a}_n=\tr(C_nD_n)/\tr(C_n^2)$ in \eqref{eq:an-def}, this can be rewritten as
\begin{equation}\label{eq:Dperp-identity}
\tr\big((D_n^\perp)^2\big)=\tr(D_n^2)-\frac{\tr(C_nD_n)^2}{\tr(C_n^2)}.
\end{equation}
By the trace approximation bounds in Lemma~\ref{lem:TA}, the remainders in \eqref{eq:TA-C2}--\eqref{eq:TA-D2} are $o\!\big(n\Delta_n^{1-2p}\big)$, hence it suffices to analyze the corresponding integrals:
\[
I_{0,n}:=\int_{-\pi}^{\pi} g_n(\lambda)^2\,d\lambda,\qquad
I_{1,n}:=\int_{-\pi}^{\pi} g_n(\lambda)h_n(\lambda)\,d\lambda,\qquad
I_{2,n}:=\int_{-\pi}^{\pi} h_n(\lambda)^2\,d\lambda,
\]
where $g_n(\lambda)=f_H(\lambda)/(1+\gamma_n f_H(\lambda))$ and $h_n(\lambda)=\dot f_H(\lambda)/(1+\gamma_n f_H(\lambda))$.
Indeed,
\[
\tr(D_n^2)=\frac{n}{2\pi}I_{2,n}+o\!\big(n\Delta_n^{1-2p}\big),\qquad
\tr(C_nD_n)=\frac{n}{2\pi}I_{1,n}+o\!\big(n\Delta_n^{1-2p}\big),\qquad
\]
$$\tr(C_n^2)=\frac{n}{2\pi}I_{0,n}+o\!\big(n\Delta_n^{1-2p}\big).$$
Consequently, \eqref{eq:Dperp-identity} implies
\[
\tr\big((D_n^\perp)^2\big)=\frac{n}{2\pi}\left(I_{2,n}-\frac{I_{1,n}^2}{I_{0,n}}\right)+o\!\big(n\Delta_n^{1-2p}\big).
\]

Fix $\eta\in(0,\pi)$ and split each $I_{k,n}=I_{k,n}^{(0)}+I_{k,n}^{(\infty)}$ over $|\lambda|\le\eta$ and $\eta<|\lambda|\le\pi$.
On $\{\eta<|\lambda|\le\pi\}$, $f_H$ and $\dot f_H$ are bounded and $a_n(\lambda)\ge1$, so $I_{k,n}^{(\infty)}=O(1)$ for $k=0,1,2$,
therefore these contributions are $O(1)$ and thus negligible compared with the leading order $\Delta_n^{1-2p}$. Moreover, although $I_{1,n}$ enters through the ratio $I_{1,n}^2/I_{0,n}$, the high-frequency parts still remain negligible:
since $I_{0,n}^{(0)}\asymp \Delta_n^{1-2p}$ and $I_{1,n}^{(0)}=O(\Delta_n^{1-2p}\ln(1/\Delta_n))$,
a perturbation $I_{k,n}=I_{k,n}^{(0)}+O(1)$ changes $I_{1,n}^2/I_{0,n}$ by at most $O(\ln(1/\Delta_n))=o(\Delta_n^{1-2p})$.
Hence $I_{2,n}-I_{1,n}^2/I_{0,n}$ can be replaced by its low-frequency counterpart up to $o(\Delta_n^{1-2p})$.

On $|\lambda|\le\eta$, set $\lambda=\Delta_n x$ and use the notation of Lemma~\ref{lem:trace-CD}:
\[
\kappa_n(x):=\frac{f_H(\Delta_n x)^2}{(1+\gamma_n f_H(\Delta_n x))^2},\qquad
b_n(x):=\frac{\dot f_H(\Delta_n x)}{f_H(\Delta_n x)}.
\]
Then
\[
I_{0,n}^{(0)}=\Delta_n\int_{|x|\le\eta/\Delta_n}\kappa_n(x)\,dx,\quad
I_{1,n}^{(0)}=\Delta_n\int_{|x|\le\eta/\Delta_n}\kappa_n(x)\,b_n(x)\,dx,
\]
$$
I_{2,n}^{(0)}=\Delta_n\int_{|x|\le\eta/\Delta_n}\kappa_n(x)\,b_n(x)^2\,dx.$$
By Lemma~\ref{lem:an-expansion}, $\mathbf{a}_n=2\ln(1/\Delta_n)+m+o(1)$ with $m=J_1/J_0$, while
$b_n(x)=2\ln(1/\Delta_n)+q(x)+o(1)$ pointwise for fixed $x\neq0$, where $q(x)=C_H-2\ln|x|$.
Moreover, $\kappa_n(x)=\Delta_n^{-2p}w(x)(1+o(1))$ pointwise with $w$ given in \eqref{eq:weight-w}, and the integrands are dominated as in Lemma~\ref{lem:verify-GSzego}.
Therefore, by dominated convergence and $\eta/\Delta_n\to\infty$,
\[
I_{2,n}^{(0)}-\frac{(I_{1,n}^{(0)})^2}{I_{0,n}^{(0)}}
=
\Delta_n^{1-2p}\left(\int_{\mathbb R} w(x)\,(q(x)-m)^2\,dx+o(1)\right)
=\Delta_n^{1-2p}\big(J_\perp+o(1)\big).
\]
Combining the low- and high-frequency parts yields \eqref{eq:Dperp-Frob}. Positivity follows since $q$ is not $w$-a.e.\ constant.
\end{proof}

Now, we will try to get the structure of $\| \cdot \|_{\mathrm{op}}/\|\cdot\|_F\to0$ for $D_n^\perp$.
\begin{lemma}\label{lem:Dperp-ratio}
Fix $H\in(3/4,1)$ and set $p:=2H-1\in(1/2,1)$. Let $\Delta_n=n^{-\alpha}$ with $\alpha\in(0,1)$,
$\gamma_n=\sigma^2\Delta_n^{p}$ and $\mathcal T_n=n\Delta_n\to\infty$.
Recall $D_n^\perp=D_n-\mathbf{a}_n C_n$ with $\mathbf{a}_n$ defined by \eqref{eq:an-def}. Then
\[
\frac{\opnorm{D_n^\perp}}{\Fnorm{D_n^\perp}}\longrightarrow 0 .
\]
\end{lemma}

\begin{proof}
From Lemma~\ref{lem:D-op} we have $\opnorm{C_n}\le 1/\gamma_n$ and $\opnorm{D_n}=O\big(\gamma_n^{-1}\log n\big)$. Now, by the triangle inequality,
\[
\opnorm{D_n^\perp}\le \opnorm{D_n}+|\mathbf{a}_n|\,\opnorm{C_n}
\le \opnorm{D_n}+\frac{|\mathbf{a}_n|}{\gamma_n}.
\]
From Lemma~\ref{lem:an-expansion} we have $\mathbf{a}_n=2\ln(1/\Delta_n)+O(1)+o(1)$, hence $|\mathbf{a}_n|=O(\log n)$ because $\Delta_n=n^{-\alpha}$.
Together with Lemma~\ref{lem:D-op} this yields
\[
\opnorm{D_n^\perp}=O\!\Big(\frac{\log n}{\gamma_n}\Big).
\]
On the other hand, Lemma~\ref{lem:Dperp-Frob} gives
\[
\Fnorm{D_n^\perp}^2=\frac{n}{2\pi}\Delta_n^{1-2p}\big(J_\perp+{o(1)}\big),
\qquad J_\perp>0,
\]
so $\Fnorm{D_n^\perp}\asymp \sqrt{n}\Delta_n^{\frac{1-2p}{2}}$.
Since $\gamma_n=\sigma^2\Delta_n^{p}$, there exists a constant $C>0$ such that
\[
\frac{\opnorm{D_n^\perp}}{\Fnorm{D_n^\perp}}
\le
C \frac{\Delta_n^{-p}\log n}{\sqrt{n}\,\Delta_n^{\frac{1-2p}{2}}}
=
C \frac{\log n}{\sqrt{n\Delta_n}}
=
C \frac{\log n}{\sqrt{\mathcal T_n}}
\longrightarrow 0,
\]
because $\mathcal T_n=n\Delta_n\to\infty$.
\end{proof}

\subsection{Joint CLT after multiplying by the rate matrix}

Define the lower-triangular matrices
\begin{equation}\label{eq: rate matrix}
M_n^{(1)}:=
\begin{pmatrix}
1 & 0\\
-\sigma\ln(\Delta_n) & 1
\end{pmatrix},
\quad
M_n^{(2)}:=
\begin{pmatrix}
1 & 0\\
-\frac{\sigma\mathbf{a}_n }{2} & 1
\end{pmatrix},
\quad
M_n:=M_n^{(2)}M_n^{(1)}
=
\begin{pmatrix}
1 & 0\\
-\sigma\ln(\Delta_n)-\frac{\sigma \mathbf{a}_n}{2} & 1
\end{pmatrix}.
\end{equation}
Recall from Lemma~\ref{lem:explicit-linear} that $S_{H,n}=\sigma\ln(\Delta_n)\,S_{\sigma,n}+R_{H,n}$.
Then
\[
M_n
\binom{S_{\sigma,n}}{S_{H,n}}
=
\binom{S_{\sigma,n}}{R_{H,n}^\perp}.
\]

\begin{proposition}\label{prop:projected-joint-clt}
Assume $\mathcal T_n=n\Delta_n\to\infty$. Define
\[
\Xi_n:=
\frac{1}{\sqrt{\mathcal T_n}}\,M_n
\binom{S_{\sigma,n}}{S_{H,n}}
=
\binom{\frac{S_{\sigma,n}}{\sqrt{\mathcal T_n}}}{\frac{R_{H,n}^\perp}{\sqrt{\mathcal T_n}}}.
\]
Then, as $n\to\infty$,
\[
\Xi_n\ \Rightarrow\ \mathcal N\!\left(0,\ I^\perp\right),
\qquad
I^\perp:=
\begin{pmatrix}
\frac{\sigma^2}{\pi}J_0(H,\sigma) & 0\\
0 & \frac{\sigma^4}{4\pi}J_\perp(H,\sigma)
\end{pmatrix},
\qquad
J_\perp=J_2-\frac{J_1^2}{J_0}>0.
\]
\end{proposition}

\begin{proof}
By Proposition~\ref{prop:sigma-clt},
\[
\frac{S_{\sigma,n}}{\sqrt{\mathcal T_n}}\Rightarrow \mathcal N\!\left(0,\frac{\sigma^2}{\pi}J_0\right).
\]
Moreover,
\[
R_{H,n}^\perp=\frac{\gamma_n}{2}\Big(Z_{n}^\top D_n^\perp Z_{n}-\Tr(D_n^\perp)\Big).
\]
By Lemma~\ref{lem:Dperp-ratio} and Lemma~\ref{lem:QF-CLT},
\[
\frac{Z_{n}^\top D_n^\perp Z_{n}-\Tr(D_n^\perp)}{\sqrt{2}\Fnorm{D_n^\perp}}\Rightarrow \mathcal N(0,1).
\]
Using $\Fnorm{D_n^\perp}^2=\frac{n}{2\pi}\Delta_n^{1-2p}(J_\perp+o(1))$ and $\gamma_n=\sigma^2\Delta_n^p$ yields
\[
\frac{R_{H,n}^\perp}{\sqrt{\mathcal T_n}}\Rightarrow \mathcal N\!\left(0,\frac{\sigma^4}{4\pi}J_\perp\right).
\]
Now, we consider the covariance part. For any fixed $u=(u_1,u_2)\in\R^2$,
\[
u_1\frac{S_{\sigma,n}}{\sqrt{\mathcal T_n}}+u_2\frac{R_{H,n}^\perp}{\sqrt{\mathcal T_n}}
=
\frac{\gamma_n}{\sqrt{\mathcal T_n}}
\Big(Z_{n}^\top M_{n,u}Z_{n}-\Tr(M_{n,u})\Big),
\qquad
M_{n,u}:=u_1\sigma^{-1}C_n+u_2\tfrac12 D_n^\perp.
\]
Since $\Tr(C_nD_n^\perp)=0$ by construction, we have
\[
\Fnorm{M_{n,u}}^2
=
u_1^2\sigma^{-2}\Fnorm{C_n}^2+\frac{u_2^2}{4}\Fnorm{D_n^\perp}^2,
\]
and $\opnorm{M_{n,u}}\le \sigma^{-1}|u_1|\opnorm{C_n}+2^{-1}|u_2|\opnorm{D_n^\perp}$.
Since $\Fnorm{M_{n,u}}^2$ is a nonnegative sum of the two terms above, we obtain
$\opnorm{M_{n,u}}/\Fnorm{M_{n,u}}\to 0$ from $\opnorm{C_n}/\Fnorm{C_n}\to0$ (verified in Proposition~\ref{prop:sigma-clt})
and Lemma~\ref{lem:Dperp-ratio}. Applying Lemma~\ref{lem:QF-CLT} to $M_{n,u}$ yields a one-dimensional normal limit for every $u$,
hence joint convergence by Cram\'er--Wold.

Finally, Lemma~\ref{lem:orthogonality} gives $\Cov(S_{\sigma,n},R_{H,n}^\perp)=0$ for each $n$, so the limiting covariance is diagonal.
\end{proof}
\begin{remark}
In the matrix $M_n$, we can change $\mathbf{a_n}$ directly with $2\ln\!\Big(\frac{1}{\Delta_n}\Big)+\frac{J_1}{J_0}$ and obtain the same result.
\end{remark}

\section{LAN Property---Proof of Theorem \ref{thm:LAN-rank2}}\label{sec:LAN}
In this section, we combine the Central Limit Theorem for the projected scores (Proposition~\ref{prop:projected-joint-clt}) with the local expansion of the Gaussian likelihood to prove the main LAN theorem.
\subsection{Matrix Taylor Expansions}
To establish the Local Asymptotic Normality, we need to analyze the asymptotic behavior of the log-likelihood ratio. This ratio involves non-linear matrix functions of the covariance matrix, specifically the logarithm of the determinant and the matrix inverse.
To handle these terms, we rely on second-order Taylor expansions with respect to the perturbation matrix. The following lemma provides the fundamental matrix inequalities required to expand these terms and strictly control the remainder terms using the operator norm and the Frobenius norm.

\begin{lemma}\label{lem:matrix-expansions}
Let $S$ be a real symmetric $n\times n$ matrix with $\opnorm{S}\le \frac12$, we consider the Log-determinant and inverse of the matrix $I_n+S$ and have the following results: 

\begin{enumerate}
\item $\log\det(I_n+S)=\Tr(S)-\frac12\Tr(S^2)+R_{\log}(S)$, where the remainder satisfies $|R_{\log}(S)|\le C\,\opnorm{S}\,\Tr(S^2)$.

\item $(I_n+S)^{-1}=I_n-S+S^2+R_{\mathrm{inv}}(S) \qquad R_{\mathrm{inv}}(S):=\sum_{k\ge 3}(-1)^kS^k$,
and the remainder satisfies
\[
\opnorm{R_{\mathrm{inv}}(S)}\le C\,\opnorm{S}^3,
\qquad
\Tr\!\big(R_{\mathrm{inv}}(S)^2\big)\le C\,\Tr(S^6).
\]
\end{enumerate}
Here $C>0$ is a universal constant.
\end{lemma}

\begin{proof}
Let $\lambda_1,\dots,\lambda_n$ be the eigenvalues of $S$. Since $S$ is symmetric and $\opnorm{S}\le 1/2$,
we have $|\lambda_i|\le 1/2$ for all $i$.

\smallskip
\noindent\textbf{(1) Log-determinant.}
For $|x|\le 1/2$, Taylor's theorem gives
\[
\log(1+x)=x-\frac12x^2 + x^3\rho(x)
\quad\text{with}\quad |\rho(x)|\le C.
\]
Summing over eigenvalues yields
\[
\log\det(I_n+S)=\sum_{i=1}^n\log(1+\lambda_i)
=\Tr(S)-\frac12\Tr(S^2)+\sum_{i=1}^n \lambda_i^3\rho(\lambda_i).
\]
Therefore,
\[
|R_{\log}(S)|
\le C\sum_{i=1}^n |\lambda_i|^3
\le C\,\max_i|\lambda_i| \sum_{i=1}^n \lambda_i^2
\le C\,\opnorm{S}\,\Tr(S^2).
\]

\smallskip
\noindent\textbf{(2) Inverse.}
For $|x|\le 1/2$,
\[
(1+x)^{-1}=1-x+x^2+\sum_{k\ge 3}(-1)^k x^k.
\]
By functional calculus this gives
\[
(I_n+S)^{-1}=I_n-S+S^2+\sum_{k\ge 3}(-1)^kS^k,
\]
so the stated expression for $R_{\mathrm{inv}}(S)$ holds. Moreover,
\[
\opnorm{R_{\mathrm{inv}}(S)}
\le \sum_{k\ge 3}\opnorm{S}^k
=\frac{\opnorm{S}^3}{1-\opnorm{S}}
\le 2\,\opnorm{S}^3.
\]
Since $R_{\mathrm{inv}}(S)$ is a polynomial in $S$, it commutes with $S$ and we may write
$R_{\mathrm{inv}}(S)=S^3Q(S)$ with $Q(S):=\sum_{k\ge 0}(-1)^k S^k=(I_n+S)^{-1}$.
Hence
\[
R_{\mathrm{inv}}(S)^2 = Q(S)^2 S^6,
\qquad
\Tr\!\big(R_{\mathrm{inv}}(S)^2\big)=\Tr\!\big(Q(S)^2S^6\big)
\le \opnorm{Q(S)}^2\,\Tr(S^6).
\]
Finally, $\opnorm{Q(S)}=\opnorm{(I_n+S)^{-1}}\le (1-\opnorm{S})^{-1}\le 2$, which yields the claimed bound.
\end{proof}

\subsection{Final Proof of Theorem~\ref{thm:LAN-rank2}}

Fix $h\in\R^2$ and set $\theta_h:=(\sigma_{h}, H_{h})=\theta+h r_n^{-1}$. Write $V:=V_n(\theta)$ and $V_h:=V_n(\theta_h)$.
Under $\mathbb P_{\theta}$, the whitened vector $Z_{n}:=V^{-1/2}X_n$ is $\mathcal{N}(0,I_n)$.
Define the symmetric perturbation matrix
\[
S_n(h):=V^{-1/2}(V_h-V)V^{-1/2}.
\]
A standard Gaussian likelihood identity gives the exact log-likelihood ratio representation
\begin{equation}\label{eq:LLR-S}
\ell_n(\theta_h)-\ell_n(\theta)
=
-\frac12\log\det(I_n+S_n(h))
-\frac12\,Z_{n}^\top\Big((I_n+S_n(h))^{-1}-I_n\Big)Z_{n}.
\end{equation}

\smallskip
First, we verify the condition required for the matrix expansion, namely that the operator norm of the perturbation matrix $S_n(h)$ vanishes asymptotically.
Recall $V_n(\theta)=\Delta_nA_n(\theta)$ with $A_n(\theta):=I_n+\gamma_n(\theta)T_n(H)$, hence
$V^{-1/2}=\Delta_n^{-1/2}A_n(\theta)^{-1/2}$ and
\[
S_n(h)=A_n(\theta)^{-1/2}\big(A_n(\theta_h)-A_n(\theta)\big)A_n(\theta)^{-1/2}.
\]
Since $r_n^{-\top}=(1/\sqrt{\mathcal T_n})M_n$ and, by Lemma~\ref{lem:an-expansion}, the matrix $M_n$ is bounded (recall $\mathbf{a}_n$ involves $\ln(1/\Delta_n)$ which cancels with the scaling), we have $\sigma_h-\sigma=O(\mathcal T_n^{-1/2})$ and {$H_h-H=O(\mathcal T_n^{-1/2})$}.

A first-order Taylor expansion yields
\[
A_n(\theta_h)-A_n(\theta)
=
(\gamma_h-\gamma)T_n(H)+\gamma(H_h-H)\dot T_n(H)+  R_{A,n},
\]
where {$\gamma_{h}:=\gamma(\theta_{h})$} and $ R_{A,n}$ collects the second-order Taylor terms. Moreover, this remainder is negligible after sandwiching:
using $\|\ddot T_n(H)\|_{\op}=O(\log^2 n)$ (coming from the $\log^2|2\sin(\lambda/2)|$ singularity in $\partial_H^2\log f_H$) and
 {$H_h-H=O\big(\mathcal T_n^{-1/2}|\log\Delta_n|^{-1}\big)$, since $|\log\Delta_n|\asymp \log n$ and $\mathcal T_n\to\infty$), we obtain
$$\|A_n(\theta)^{-1/2}  R_{A,n}A_n(\theta)^{-1/2}\|_{\op}=o(1).$$
Sandwiching by $A_n(\theta)^{-1/2}$ gives
\[
S_n(h)
=
(\gamma_h-\gamma)C_n(\theta)+\gamma(H_h-H)D_n(\theta)+R_{S,n},
\]
with 
$$
C_n(\theta)=A_n(\theta)^{-1/2}T_n(H)A_n(\theta)^{-1/2},\,D_n(\theta)=A_n(\theta)^{-1/2}\dot T_n(H)A_n(\theta)^{-1/2}
$$
and $\|R_{S,n}\|_{\op}=o(1)$ uniformly for fixed $h$. Using $\opnorm{C_n(\theta)}\le 1/\gamma$ and $\opnorm{D_n(\theta)}=O(\gamma^{-1}\log n)$ (by Lemma~\ref{lem:D-op}),
together with $(\gamma_h-\gamma)=O(\gamma\mathcal T_n^{-1/2})$ and $H_h-H=O\big(\mathcal T_n^{-1/2}|\log\Delta_n|^{-1}\big)$,
we obtain
\[
\opnorm{(\gamma_h-\gamma)C_n(\theta)}=O(\mathcal T_n^{-1/2}),
\qquad
\opnorm{\gamma(H_h-H)D_n(\theta)}=O(\mathcal T_n^{-1/2}\log n),
\]
and the second-order remainder satisfies $\opnorm{R_{S,n}}=o(1)$.
Since $\mathcal T_n=n\Delta_n=n^{1-\alpha}$ and $\log n/\sqrt{\mathcal T_n}\to 0$, it follows that $\opnorm{S_n(h)}\to 0$.

\smallskip
Next, having established that $\opnorm{S_n(h)} \to 0$, we apply Lemma~\ref{lem:matrix-expansions} to obtain a second-order expansion of the log-likelihood ratio \eqref{eq:LLR-S}.
For $n$ large enough we have $\opnorm{S_n(h)}\le 1/2$, hence
\begin{align}
\ell_n(\theta_h)-\ell_n(\theta)
&=
\frac12\big(Z_{n}^\top S_n(h) Z_{n}-\Tr S_n(h)\big)
-\frac14\Tr\!\big(S_n(h)^2\big)\nonumber\\
&\quad
-\frac12\big(Z_{n}^\top S_n(h)^2 Z_{n}-\Tr(S_n(h)^2)\big)
-\frac12R_{\log}(S_n(h))
-\frac12 Z_{n}^\top R_{\mathrm{inv}}(S_n(h))Z_{n}.
\label{eq:LLR-expansion}
\end{align}

\smallskip
Then, we identify the limits of the principal linear and quadratic terms in this expansion.
Let $\delta_n:=\theta_h-\theta=h r_n^{-1}$ (viewed as a vector perturbation). By the Fr\'echet differentiability of $\theta\mapsto V_n(\theta)$,
\[
S_n(h)=\sum_{i=1}^2 \delta_{n,i}\,M_{i,n}+R_{n},
\qquad
M_{i,n}:=V^{-1/2}\,\partial_{\theta_i}V_n(\theta)\,V^{-1/2},
\]
where the remainder satisfies $\Tr(R_n^2)=o(1)$ (the Taylor remainder is quadratic in $\delta_n$, and after sandwiching its
Frobenius norm is $O(\|\delta_n\|^2)$; since $\|\delta_n\|=O(\mathcal T_n^{-1/2})$, this gives $\Tr(R_n^2)=O(\mathcal T_n^{-1})\to 0$).
Therefore,
\[
\frac12\big(Z_{n}^\top S_n(h)Z_{n}-\Tr S_n(h)\big)
=
\sum_{i=1}^2 \delta_{n,i}\,\partial_{\theta_i}\ell_n(\theta)
+\frac12\big(Z_{n}^\top R_n Z_{n}-\Tr (R_{n}) \big).
\]
The last term is $o_{\mathbb P}(1)$ since $\Var(Z_{n}^\top R_n Z_{n}-\Tr (R_n))=2\Tr(R_n^2)=o(1)$.
Hence, using $\delta_n^\top \nabla \ell_n(\theta) = h^\top (r_n^{-\top} \nabla \ell_n(\theta))$, we have
\begin{equation}\label{eq:linear-term}
\frac12\big(Z_{n}^\top S_n(h)Z_{n}-\Tr S_n(h)\big)
=
\delta_n^\top\nabla\ell_n(\theta)+o_{\mathbb P}(1)
=
h^\top\Xi_n+o_{\mathbb P}(1),
\end{equation}
where $\Xi_n=r_n^{-\,\top}\nabla\ell_n(\theta)$.

Next, using $\Tr((A+B)^2)=\Tr(A^2)+2\Tr(AB)+\Tr(B^2)$ and Cauchy--Schwarz,
\[
\Tr(S_n(h)^2)=\Tr\!\Big(\Big(\sum_{i=1}^2\delta_{n,i}M_{i,n}\Big)^2\Big)+o(1).
\]
In the Gaussian covariance model one has the identity (Wick formula)
\[
\Cov_{\theta}\big(\partial_{\theta_i}\ell_n(\theta),\partial_{\theta_j}\ell_n(\theta)\big)
=\frac12\Tr(M_{i,n}M_{j,n})=: \mathcal I_{ij,n}(\theta),
\]
so $\mathcal I_n(\theta)=(\mathcal I_{ij,n}(\theta))_{i,j}$ is both the Fisher information and the covariance of the score.
Therefore,
\[
\frac14\Tr(S_n(h)^2)
=
\frac12\,\delta_n^\top \mathcal I_n(\theta)\,\delta_n+o(1)
=
\frac12\,h^\top\Big(r_n^{-\,\top}\mathcal I_n(\theta)r_n^{-1}\Big)h+o(1).
\]
Moreover, in the Gaussian covariance model $\Cov_{\theta}(\nabla\ell_n(\theta))=\mathcal I_n(\theta)$, and the same trace asymptotics used in Proposition~\ref{prop:projected-joint-clt} yield
$r_n^{-\,\top}\mathcal I_n(\theta)r_n^{-1}\to I^\perp$.
Consequently,
\begin{equation}\label{eq:TrS2-limit}
\frac14\Tr(S_n(h)^2)=\frac12\,h^\top I^\perp h+o(1),
\qquad\text{in particular }\Tr(S_n(h)^2)=O(1).
\end{equation}

\smallskip
Finally, we control the remaining high-order error terms in \eqref{eq:LLR-expansion} and show they converge to zero in probability.
Using $\Var(Z_{n}^\top A Z_{n}-\Tr (A))=2\Tr(A^2)$ and $\Tr(S^4)\le \opnorm{S}^2\Tr(S^2)$, we obtain
\[
\Var\big(Z_{n}^\top S_n(h)^2 Z_{n}-\Tr(S_n(h)^2)\big)=2\Tr(S_n(h)^4)
\le 2\,\opnorm{S_n(h)}^2\,\Tr(S_n(h)^2)\to 0,
\]
by the results established above, hence $Z_{n}^\top S_n(h)^2 Z_{n}-\Tr(S_n(h)^2)=o_{\mathbb P}(1)$.
Moreover, Lemma~\ref{lem:matrix-expansions} and \eqref{eq:TrS2-limit} give
$$
|R_{\log}(S_n(h))|
\le C\,\opnorm{S_n(h)}\,\Tr(S_n(h)^2)=o(1).
$$
Also, write $Z_{n}^\top R_{\mathrm{inv}}(S_n(h))Z_{n}=\Tr(R_{\mathrm{inv}}(S_n(h)))+\big(Z_{n}^\top R_{\mathrm{inv}}(S_n(h))Z_{n}-\Tr(R_{\mathrm{inv}}(S_n(h)))\big)$.
Since $R_{\mathrm{inv}}(S)=\sum_{k\ge 3}(-1)^kS^k$, we have
$|\Tr(R_{\mathrm{inv}}(S))|\le \sum_{k\ge 3}|\Tr(S^k)|\le C\,\opnorm{S}\Tr(S^2)$, hence $|\Tr(R_{\mathrm{inv}}(S_n(h)))|=o(1)$.
Furthermore,
\begin{align*}
    \Var\big(Z_{n}^\top R_{\mathrm{inv}}(S_n(h))Z_{n}-\Tr(R_{\mathrm{inv}}(S_n(h)))\big)
&=2\Tr(R_{\mathrm{inv}}(S_n(h))^2)\\
&\le C\,\Tr(S_n(h)^6)
\le C\,\opnorm{S_n(h)}^4\,\Tr(S_n(h)^2)\to 0,
\end{align*}

so $Z_{n}^\top R_{\mathrm{inv}}(S_n(h))Z_{n}=o_{\mathbb P}(1)$.

\smallskip
Combining \eqref{eq:LLR-expansion}, \eqref{eq:linear-term} and \eqref{eq:TrS2-limit} yields
\[
\ell_n(\theta+h r_n^{-1})-\ell_n(\theta)
=
h^\top\Xi_n-\frac12 h^\top I^\perp h+o_{\mathbb P}(1),
\]
as claimed.
\begin{remark}[Equivalent LAN in the original $(\sigma,H)$-coordinates]\label{rem:LAN-original}
LAN is invariant under deterministic invertible linear reparametrizations of the local parameter.
Let $L$ be any fixed invertible $2\times2$ matrix and define
\[
\Xi_n^{(L)}:=L^\top \Xi_n,\qquad r_n^{(L)}:=r_nL^{-1},\qquad I^{(L)}:=L^\top I^\perp L.
\]
Then the LAN expansion holds equivalently with $(\Xi_n^{(L)},I^{(L)},r_n^{(L)})$.
In particular, choosing $L$ to map the orthogonalized coordinates back to the original parameter basis
typically yields a symmetric but non-diagonal information matrix, matching standard continuous-time LAN presentations.
\end{remark}

\begin{remark}
Since $\mathbf{a}_n=2\ln(1/\Delta_n)+m(H,\sigma)+o(1)$ presented in \eqref{eq:an-2nd}, in this theorem we can replace $\mathbf{a}_n$ by $\tilde {\mathbf{a}}_n:=2\ln(1/\Delta_n)+m(H,\sigma)$ in the rate matrix $r_n$
(without changing the LAN expansion), because the induced change in the local perturbation
$\theta_0+r_n^{-1}h$ is $o(\mathcal T_n^{-1/2})$.
\end{remark}

% ---- in the body ----
\section{Simulation Study}\label{sec:simulation}

We illustrate the asymptotic behavior of the rate-matrix transformed score
\[
\Xi_n
:=\frac{1}{\sqrt{\mathcal T_n}}\,M_n\,\nabla\ell_n(\theta_0)
=
\binom{S_{\sigma,n}/\sqrt{\mathcal T_n}}{R_{H,n}^\perp/\sqrt{\mathcal T_n}},
\qquad \mathcal T_n:=n\Delta_n,
\]
predicted by the projected-score CLT (Proposition~\ref{prop:projected-joint-clt}) and the LAN theorem (Theorem~\ref{thm:LAN-rank2}).
Throughout we fix $(H,\sigma)=(0.80,1.00)$ and use the explicit integral constants $J_0(H,\sigma)$ and
$J_\perp(H,\sigma)$ (Appendix~\ref{app:J012-closed}).

\subsection{Projected (full-rank) limit}
Proposition~\ref{prop:projected-joint-clt} yields the diagonal limiting covariance
\[
I^\perp=
\begin{pmatrix}
\frac{\sigma^2}{\pi}J_0(H,\sigma) & 0\\[1mm]
0 & \frac{\sigma^4}{4\pi}J_\perp(H,\sigma)
\end{pmatrix}.
\]
For $(H,\sigma)=(0.80,1.00)$ we compute
\[
J_0=0.2820,\quad J_\perp=34.1772,
\quad
\Var(\Xi_{\sigma})=\frac{\sigma^2}{\pi}J_0=0.0897,
\quad
\Var(\Xi_{H}^{\perp})=\frac{\sigma^4}{4\pi}J_\perp=2.7197.
\]
The simulated joint distribution of $\Xi_n$ is well-approximated by the corresponding bivariate Gaussian, as shown below.

\begin{figure}[H]
\centering
\includegraphics[width=0.56\textwidth]{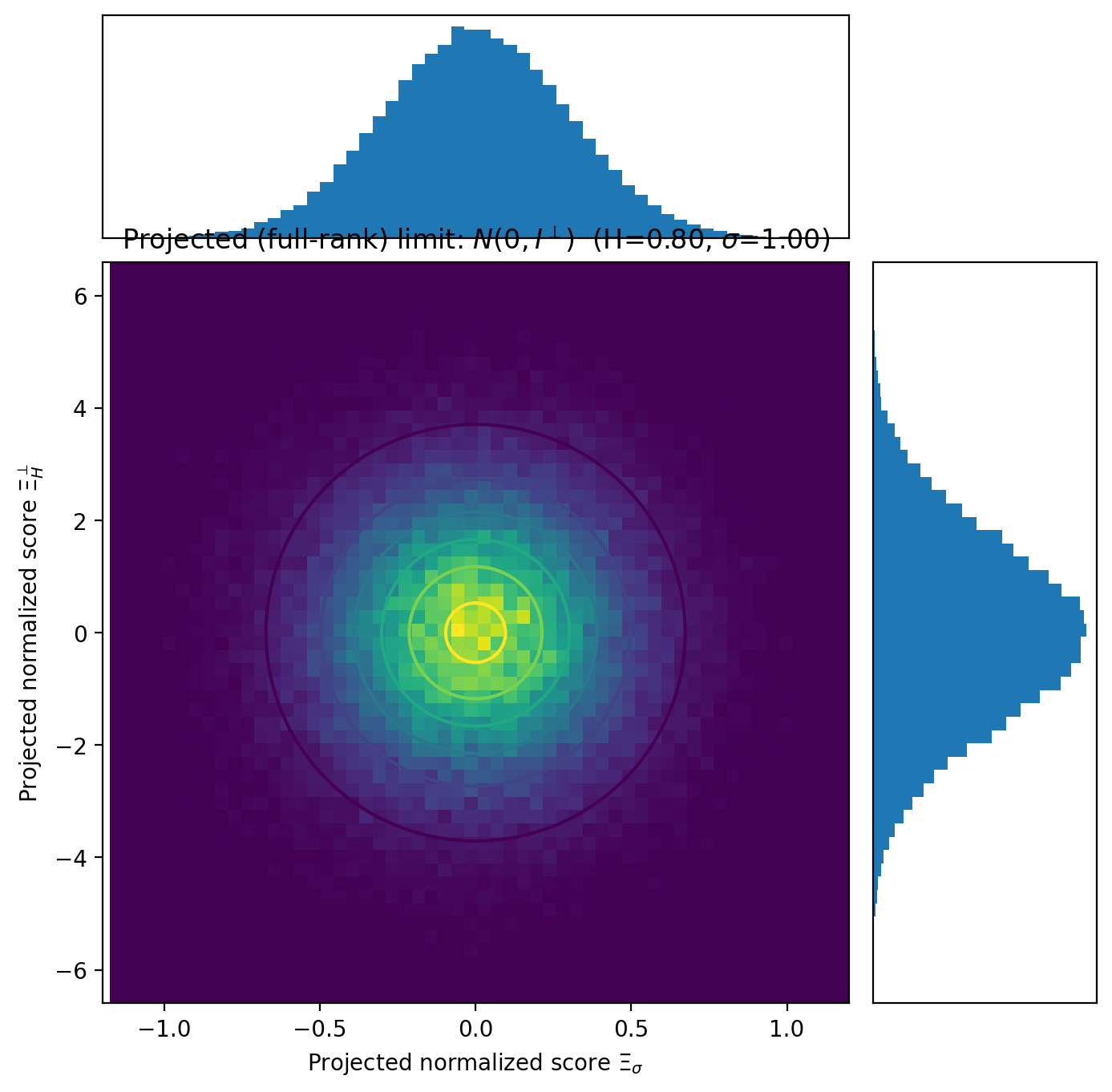}
\caption{Projected/orthogonalized score $\Xi_n$ (simulation) with theoretical Gaussian contours.}
\label{fig:sim-proj-2d}
\end{figure}

\begin{figure}[H]
\centering
\includegraphics[width=0.56\textwidth]{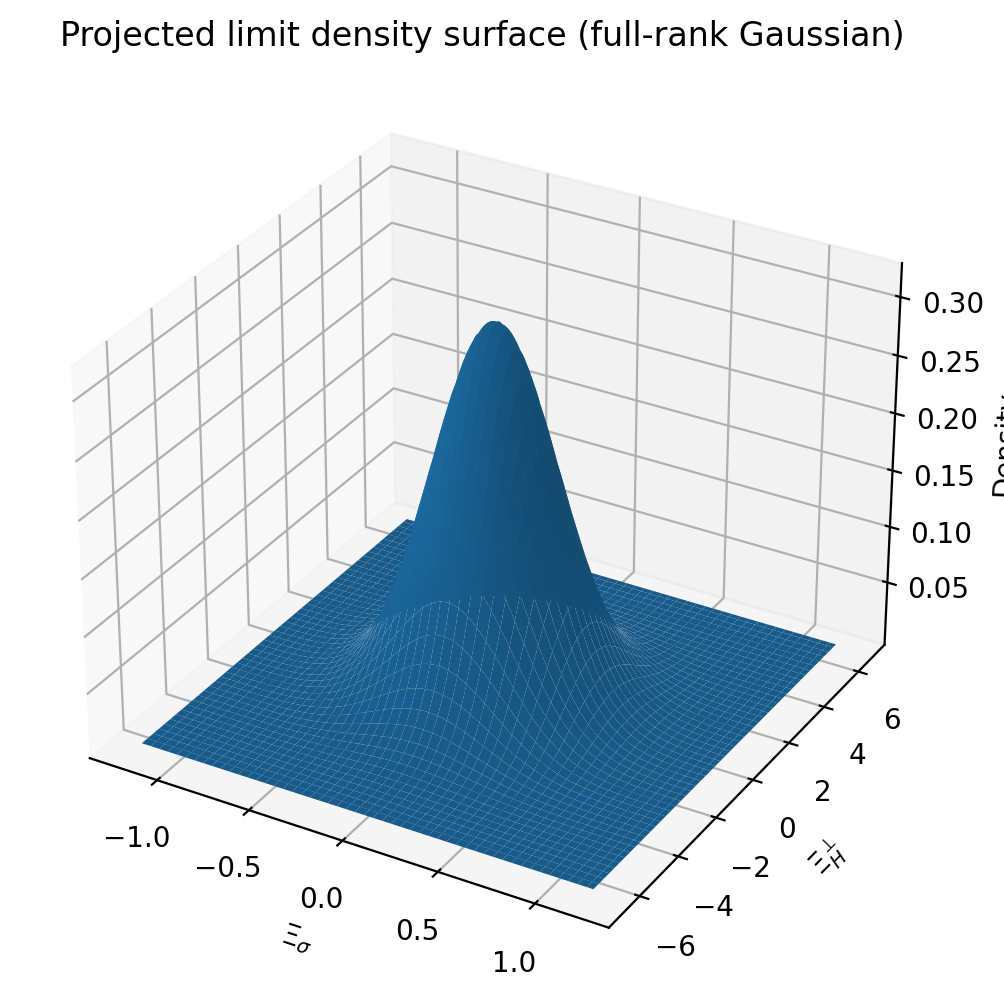}
\caption{Limiting Gaussian surface density of the projected score.}
\label{fig:sim-proj-3d}
\end{figure}

\subsection{Why projection is necessary (rank-$1$ degeneracy without it).}
If one only removes the explicit linear term in the $H$-score but does not project $D_n$ onto the orthogonal complement of $C_n$,
the limiting covariance becomes singular (rank $1$), as explained in Appendix~\ref{app:two-M}. This degeneracy is visualized below.

\begin{figure}[H]
\centering
\includegraphics[width=0.56\textwidth]{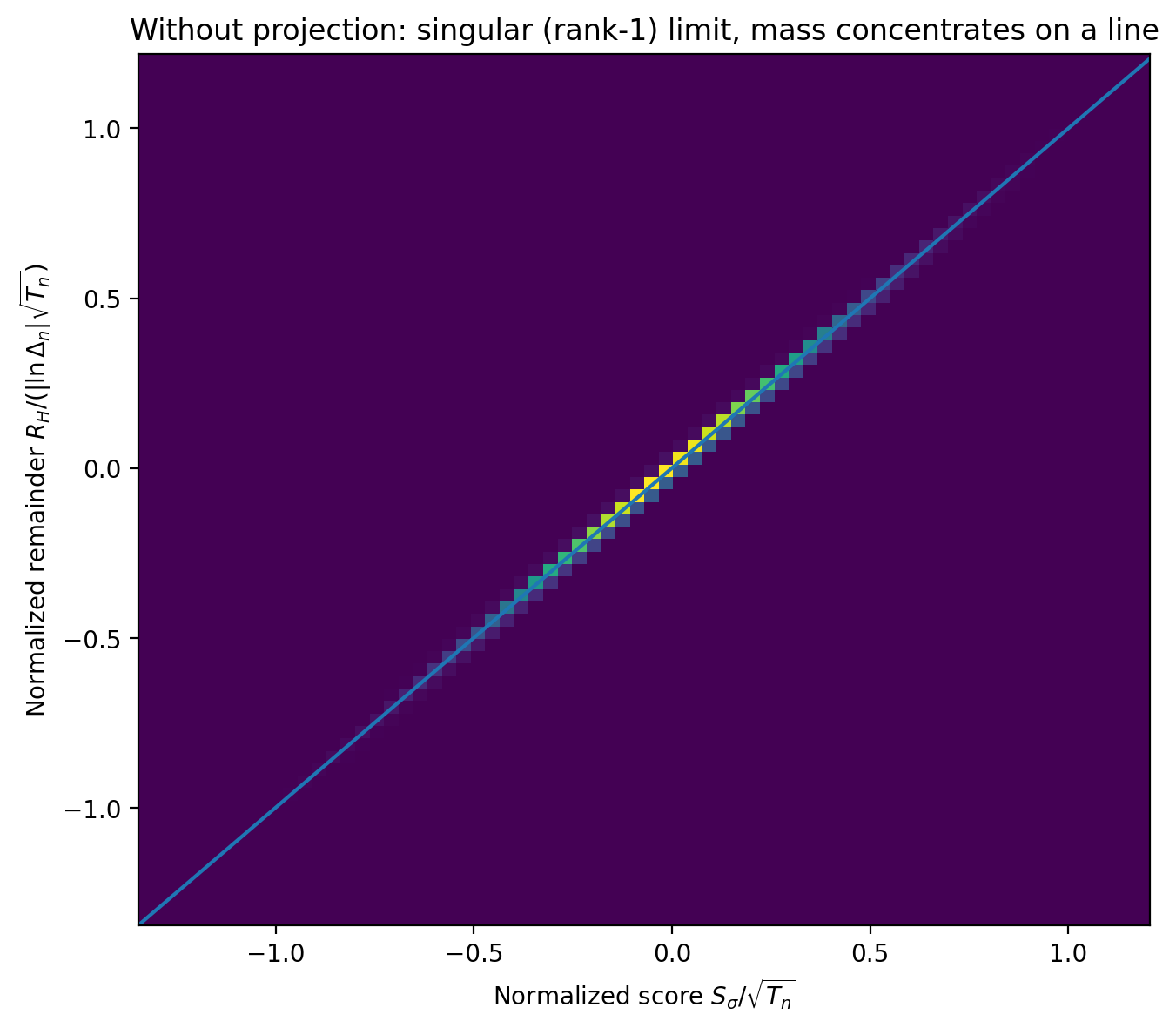}
\caption{Without projection: singular (rank-$1$) limit; the normalized pair concentrates near a line.}
\label{fig:sim-rank1}
\end{figure}

\appendix
%%%%%%%%%%%%%%%%%%%%%%%%%%%%%%%%%%%%%%%%%%%%%%%%%%%%%%%%%%%%%%%%%%%%%%%%%%%%%%%
\section{Closed-form expressions for $J_0,J_1,J_2$}\label{app:J012-closed}
%%%%%%%%%%%%%%%%%%%%%%%%%%%%%%%%%%%%%%%%%%%%%%%%%%%%%%%%%%%%%%%%%%%%%%%%%%%%%%%
\setcounter{equation}{0} % 重置公式计数器为0
\renewcommand{\theequation}{A.\arabic{equation}} 
Recall the weight function (see \eqref{eq:weight-w})
\[
w(x):=\left(\frac{c_H|x|^{-p}}{1+\sigma^2c_H|x|^{-p}}\right)^2,
\qquad p=2H-1\in(1/2,1),
\]
and the constants (see \eqref{eq:J0}--\eqref{eq:J2})
\begin{align*}
J_0(H,\sigma)&:=\int_{\R} w(x)\,dx,\\
J_1(H,\sigma)&:=\int_{\R} w(x)\Bigl(C_H-2\ln|x|\Bigr)\,dx,\\
J_2(H,\sigma)&:=\int_{\R} w(x)\Bigl(C_H-2\ln|x|\Bigr)^2\,dx.
\end{align*}
Throughout this appendix, set
\[
A:=\sigma^2 c_H>0,\qquad L_A:=\ln A.
\]

\begin{lemma}\label{lem:master-integral}
Define, for $r$ in a neighborhood of $0$,
\[
I(r):=\int_0^\infty \left(\frac{A\,x^{-p}}{1+A\,x^{-p}}\right)^2 x^{r}\,dx
=\int_0^\infty \frac{A^2 x^{r-2p}}{(1+A x^{-p})^2}\,dx.
\]
Then
\begin{equation}\label{eq:I-gamma}
I(r)=\frac{1}{p}\,A^{(r+1)/p}\,
\Gamma\!\left(\frac{r+1}{p}\right)\Gamma\!\left(2-\frac{r+1}{p}\right),
\end{equation}
whenever $0<\frac{r+1}{p}<2$. Moreover,
\[
I'(0)=\int_0^\infty w(x)\,\ln x\,dx,\qquad
I''(0)=\int_0^\infty w(x)\,(\ln x)^2\,dx.
\]
\end{lemma}

\begin{proof}
Use the change of variables $t=A x^{-p}$, i.e.\ $x=(A/t)^{1/p}$.
A direct computation yields
\[
I(r)=\frac{1}{p}A^{(r+1)/p}\int_0^\infty \frac{t^{1-(r+1)/p}}{(1+t)^2}\,dt
=\frac{1}{p}A^{(r+1)/p}B\!\left(2-\frac{r+1}{p},\,\frac{r+1}{p}\right),
\]
which gives \eqref{eq:I-gamma} by the identity $B(x,y)=\Gamma(x)\Gamma(y)/\Gamma(x+y)$ and $\Gamma(2)=1$.
Finally, differentiating under the integral sign (justified by dominated convergence for $r$ near $0$) gives
$I'(0)=\int_0^\infty w(x)\ln x\,dx$ and $I''(0)=\int_0^\infty w(x)(\ln x)^2\,dx$.
\end{proof}

\subsection{Closed form for $J_0$}
Since $w$ is even, $J_0=2\int_0^\infty w(x)\,dx=2I(0)$. Lemma~\ref{lem:master-integral} with $r=0$ yields
\begin{equation}\label{eq:J0-gamma-closed}
J_0(H,\sigma)=\frac{2}{p{\sigma^{4}}}\,A^{1/p}\,
\Gamma\!\left(\frac1p\right)\Gamma\!\left(2-\frac1p\right)
=\frac{2}{p}\,{\sigma^{2/p-4}c^{1/p}_{H}} \,
\Gamma\!\left(\frac1p\right)\Gamma\!\left(2-\frac1p\right).
\end{equation}

\subsection{Closed form for $J_1$ and $J_2$}
Introduce the one-sided logarithmic moments
\[
K_1:=\int_0^\infty w(x)\,\ln x\,dx=I'(0),\qquad
K_2:=\int_0^\infty w(x)\,(\ln x)^2\,dx=I''(0).
\]
Using evenness again,
\begin{equation}\label{eq:J1J2-via-K}
J_1=C_H J_0-4K_1,\qquad
J_2=C_H^2 J_0-8C_H K_1+8K_2.
\end{equation}

To express $K_1,K_2$ explicitly, differentiate $\ln I(r)$ at $r=0$. From \eqref{eq:I-gamma},
\[
\ln I(r)= -\ln p+\frac{r+1}{p}\ln A
+\ln\Gamma\!\left(\frac{r+1}{p}\right)
+\ln\Gamma\!\left(2-\frac{r+1}{p}\right).
\]
Let $\psi=\Gamma'/\Gamma$ denote the digamma function and $\psi_1=\psi'$ the trigamma function. Then
\begin{align}
\frac{I'(0)}{I(0)}
&=\frac{1}{p}\Big(L_A+\psi(1/p)-\psi(2-1/p)\Big),\label{eq:I1-over-I0-closed}\\
\frac{I''(0)}{I(0)}
&=\Big(\frac{I'(0)}{I(0)}\Big)^2+\frac{1}{p^2}\Big(\psi_1(1/p)+\psi_1(2-1/p)\Big).\label{eq:I2-over-I0-closed}
\end{align}
Since $I(0)=J_0/2$, we obtain
\begin{equation}\label{eq:K1K2-closed}
K_1=\frac{J_0}{2}\cdot \frac{1}{p}\Big(L_A+\psi(1/p)-\psi(2-1/p)\Big),
\end{equation}
\[
K_2=\frac{J_0}{2}\left[
\Big(\frac{I'(0)}{I(0)}\Big)^2+\frac{1}{p^2}\Big(\psi_1(1/p)+\psi_1(2-1/p)\Big)
\right].
\]
Combining \eqref{eq:J1J2-via-K} and \eqref{eq:K1K2-closed} gives the closed forms:
\begin{equation}\label{eq:J1-closed-gamma}
J_1(H,\sigma)
=
J_0(H,\sigma)\left[
C_H-\frac{2}{p}\Big(L_A+\psi(1/p)-\psi(2-1/p)\Big)
\right],
\end{equation}
and
\begin{equation}\label{eq:J2-closed-gamma}
J_2(H,\sigma)
=
J_0(H,\sigma)\left[
\Big(C_H-\frac{2}{p}\big(L_A+\psi(1/p)-\psi(2-1/p)\big)\Big)^2
+\frac{4}{p^2}\big(\psi_1(1/p)+\psi_1(2-1/p)\big)
\right].
\end{equation}
%%%%%%%%%%%%%%%%%%%%%%%%%%%%%%%%%%%%%%%%%%%%%%%%%%%%%%%%%%%%%%%%%%%%%%%%%%%%%%%
\section{Why two matrices $M_n^{(1)},M_n^{(2)}$: degeneracy without projection}\label{app:two-M}
%%%%%%%%%%%%%%%%%%%%%%%%%%%%%%%%%%%%%%%%%%%%%%%%%%%%%%%%%%%%%%%%%%%%%%%%%%%%%%%
\setcounter{equation}{0} % 重置公式计数器为0
\renewcommand{\theequation}{B.\arabic{equation}} 
This appendix explains why we use two successive transformations
\[
M_n^{(1)}=
\begin{pmatrix}1&0\\ -\sigma\ln(\Delta_n)&1\end{pmatrix},
\qquad
M_n^{(2)}=
\begin{pmatrix}1&0\\ -\frac{\sigma a_n}{2}&1\end{pmatrix},
\qquad
M_n=M_n^{(2)}M_n^{(1)},
\]
rather than a single one, and why the projection step is necessary to obtain a full-rank ($2\times2$) limiting information.

\subsection{Step 1: remove the explicit linear term in the $H$-score}
Recall the exact decomposition (see Lemma~\ref{lem:explicit-linear})
\begin{equation}\label{eq:SH-linear-app}
S_{H,n}=\sigma\ln(\Delta_n)\,S_{\sigma,n}+R_{H,n},
\qquad
R_{H,n}:=\frac{\gamma_n}{2}\Big(Z_{n}^\top D_n Z_{n}-\Tr(D_n)\Big).
\end{equation}
Therefore,
\[
M_n^{(1)}\binom{S_{\sigma,n}}{S_{H,n}}
=
\binom{S_{\sigma,n}}{R_{H,n}}.
\]
This first step produces a pair of centered Gaussian quadratic forms.

\subsection{Step 2: projection and orthogonalization}
Define the projection coefficient and orthogonalized matrix (Definition~\ref{def:an-orth})
\[
a_n:=\frac{\Tr(C_nD_n)}{\Tr(C_n^2)},
\qquad
D_n^\perp:=D_n-a_nC_n,
\qquad
\Tr(C_nD_n^\perp)=0,
\]
and the orthogonalized remainder
\[
R_{H,n}^\perp:=\frac{\gamma_n}{2}\Big(Z_{n}^\top D_n^\perp Z_{n}-\Tr(D_n^\perp)\Big)
=R_{H,n}-\frac{\sigma a_n}{2}\,S_{\sigma,n}.
\]
This corresponds exactly to the second transformation
\[
M_n^{(2)}\binom{S_{\sigma,n}}{R_{H,n}}
=
\binom{S_{\sigma,n}}{R_{H,n}^\perp}.
\]
By Wick's identity for Gaussian quadratic forms and $\Tr(C_nD_n^\perp)=0$, one has the exact orthogonality
\[
\Cov(S_{\sigma,n},R_{H,n}^\perp)=0
\qquad\text{for every }n
\]
(cf.\ Lemma~\ref{lem:orthogonality}).

\subsection{Singular covariance without the projection step}
Let $\mathcal T_n:=n\Delta_n$ and $L_n:=\ln(1/\Delta_n)=|\ln\Delta_n|$.
Consider the normalized pair after only the first transformation:
\[
U_n:=
\binom{\displaystyle \frac{S_{\sigma,n}}{\sqrt{\mathcal T_n}}}
{\displaystyle \frac{R_{H,n}}{L_n\sqrt{\mathcal T_n}}}.
\]

\begin{proposition}\label{prop:rank1-without-projection}
Recall the trace asymptotics proved in Lemma~\ref{lem:TA}:
\begin{align*}
\Tr(C_n^2)
&=\frac{n}{2\pi}\Delta_n^{1-2p}\big(J_0+o(1)\big),\\
\Tr(C_nD_n)
&=\frac{n}{2\pi}\Delta_n^{1-2p}\big(2L_nJ_0+J_1+{o(L_{n})}\big),\\
\Tr(D_n^2)
&=\frac{n}{2\pi}\Delta_n^{1-2p}\big(4L_n^2J_0+4L_nJ_1+J_2+{o(L^{2}_{n})}\big).
\end{align*}
Combining these with the operator--Frobenius negligibility (which justifies the quadratic-form CLT), the vector $U_n$ converges jointly to a centered Gaussian vector with covariance matrix
\[
\Sigma^{(1)}=
\frac{J_0(H,\sigma)}{\pi}
\begin{pmatrix}
\sigma^2 & \sigma^3\\
\sigma^3 & \sigma^4
\end{pmatrix},
\qquad \mathrm{rank}(\Sigma^{(1)})=1.
\]
Equivalently, the two components are asymptotically perfectly correlated:
\[
\Corr\!\left(\frac{S_{\sigma,n}}{\sqrt{\mathcal T_n}},\frac{R_{H,n}}{L_n\sqrt{\mathcal T_n}}\right)\longrightarrow 1.
\]
\end{proposition}

\begin{proof}
Write $Q(A):=Z_{n}^\top A Z_{n}-\Tr(A)$. Then
\[
S_{\sigma,n}=\frac{\gamma_n}{\sigma}Q(C_n),\qquad
R_{H,n}=\frac{\gamma_n}{2}Q(D_n),
\qquad \gamma_n=\sigma^2\Delta_n^{p}.
\]
By Wick's identity,
\[
\Cov(Q(A),Q(B))=2\Tr(AB),
\qquad
\Var(Q(A))=2\Tr(A^2).
\]
Hence
\[
\Var\!\left(\frac{S_{\sigma,n}}{\sqrt{\mathcal T_n}}\right)
=\frac{\gamma_n^2}{\sigma^2\mathcal T_n}\cdot 2\Tr(C_n^2)
\to \frac{\sigma^2}{\pi}J_0,
\]
and similarly
\[
\Var\!\left(\frac{R_{H,n}}{L_n\sqrt{\mathcal T_n}}\right)
=\frac{\gamma_n^2}{4L_n^2\mathcal T_n}\cdot 2\Tr(D_n^2)
\to \frac{\sigma^4}{\pi}J_0.
\]
For the covariance,
\[
\Cov\!\left(\frac{S_{\sigma,n}}{\sqrt{\mathcal T_n}},\frac{R_{H,n}}{L_n\sqrt{\mathcal T_n}}\right)
=
\frac{\gamma_n^2}{2\sigma L_n\mathcal T_n}\cdot 2\Tr(C_nD_n)
=
\frac{\gamma_n^2}{\sigma L_n\mathcal T_n}\Tr(C_nD_n).
\]
Using $\gamma_n^2=\sigma^4\Delta_n^{2p}$ and $\mathcal T_n=n\Delta_n$ together with the stated asymptotic for $\Tr(C_nD_n)$ yields
\[
\Cov\!\left(\frac{S_{\sigma,n}}{\sqrt{\mathcal T_n}},\frac{R_{H,n}}{L_n\sqrt{\mathcal T_n}}\right)
=
\frac{\sigma^4\Delta_n^{2p}}{\sigma L_n\,n\Delta_n}\cdot \frac{n}{2\pi}\Delta_n^{1-2p}\big(2L_nJ_0+{J_{1}+o(1)}\big)
\to \frac{\sigma^3}{\pi}J_0.
\]
Thus the limiting covariance matrix is exactly $\Sigma^{(1)}$, whose determinant is $0$, hence rank $1$.
\end{proof}

%%%%%%%%%%%%%%%%%%%%%%%%%%%%%%%%%%%%%%%%%%%%%%%%%%%%%%%%%%%%%%%%%%%%%%%%%%%%%%%
\section{The regime $1/2<H<3/4$}\label{app:subcritical-opF}
%%%%%%%%%%%%%%%%%%%%%%%%%%%%%%%%%%%%%%%%%%%%%%%%%%%%%%%%%%%%%%%%%%%%%%%%%%%%%%%
\setcounter{equation}{0} % 重置公式计数器为0
\renewcommand{\theequation}{C.\arabic{equation}} 
This appendix treats the \emph{subcritical long-memory} regime $1/2<H<3/4$, i.e.\ 
\[
p:=2H-1\in(0,1/2),\qquad \Delta_n=n^{-\alpha}\ \ (\alpha\in(0,1)),\qquad \gamma_n=\sigma^2\Delta_n^{p}\to0.
\]
In contrast to the supercritical case $H>3/4$ (where $f_H^2\notin L^1$), here we have
\[
f_H\in L^2([-\pi,\pi]),\qquad \dot f_H\in L^2([-\pi,\pi]),
\]
and thus the relevant Toeplitz trace functionals are of order $n$. Here we only deal with the CLT for the score function with rate matrix but not the LAN property which can be obtained easily as in the section \ref{sec:LAN}. 

\subsection{Trace approximations at order $n$}

Recall $A_n:=I_n+\gamma_nT_n(H)=T_n(a_n)$ with $a_n(\lambda):=1+\gamma_n f_H(\lambda)$ and
\[
C_n:=A_n^{-1/2}T_n(H)A_n^{-1/2},\qquad D_n:=A_n^{-1/2}\dot T_n(H)A_n^{-1/2}.
\]
Define the (triangular-array) symbols
\[
g_n(\lambda):=\frac{f_H(\lambda)}{a_n(\lambda)},\qquad 
h_n(\lambda):=\frac{\dot f_H(\lambda)}{a_n(\lambda)}.
\]

\begin{lemma}\label{lem:TA-subcritical}
Assume $1/2<H<3/4$ (equivalently $p\in(0,1/2)$) and $\gamma_n=\sigma^2\Delta_n^{p}\to0$. Then
\begin{align}
\Tr(C_n^2)
&=
\frac{n}{2\pi}\int_{-\pi}^{\pi} g_n(\lambda)^2\,d\lambda
+o(n),\label{eq:TA-sub-C2}\\
\Tr(C_nD_n)
&=
\frac{n}{2\pi}\int_{-\pi}^{\pi} g_n(\lambda)h_n(\lambda)\,d\lambda
+o(n),\label{eq:TA-sub-CD}\\
\Tr(D_n^2)
&=
\frac{n}{2\pi}\int_{-\pi}^{\pi} h_n(\lambda)^2\,d\lambda
+o(n).\label{eq:TA-sub-D2}
\end{align}
\end{lemma}

\begin{proof}
The key point is that for $p\in(0,1/2)$ we have $f_H,\dot f_H\in L^2([-\pi,\pi])$, hence
$g_n,h_n\in L^2$ uniformly because $a_n(\lambda)\ge1$. In this $L^2$-setting, classical Szeg\H{o}/Avram-type
trace theorems for products of Toeplitz matrices yield the order-$n$ trace approximation with $o(n)$ remainder,
see e.g.\ Avram~\cite{Avram88} and standard Toeplitz references such as Grenander--Szeg\H{o}~\cite{GrenanderSzego}
or B{\"o}ttcher--Silbermann~\cite{BottcherSilb}. Applying these results to the sandwiched forms corresponding to
$C_n,D_n$ (equivalently, to the symbols $g_n,h_n$) gives \eqref{eq:TA-sub-C2}--\eqref{eq:TA-sub-D2}.
\end{proof}

\begin{lemma}\label{lem:subcritical-integrals}
Under $1/2<H<3/4$ (so $p\in(0,1/2)$) and $\gamma_n\to0$,
\begin{align}
\int_{-\pi}^{\pi} g_n(\lambda)^2\,d\lambda&\longrightarrow \int_{-\pi}^{\pi} f_H(\lambda)^2\,d\lambda,\label{eq:g2-limit}\\
\int_{-\pi}^{\pi} g_n(\lambda)h_n(\lambda)\,d\lambda&\longrightarrow \int_{-\pi}^{\pi} f_H(\lambda)\dot f_H(\lambda)\,d\lambda,\label{eq:gh-limit}\\
\int_{-\pi}^{\pi} h_n(\lambda)^2\,d\lambda&\longrightarrow \int_{-\pi}^{\pi} \dot f_H(\lambda)^2\,d\lambda.\label{eq:h2-limit}
\end{align}
\end{lemma}

\begin{proof}
Since $a_n(\lambda)=1+\gamma_n f_H(\lambda)\to1$ pointwise and $a_n(\lambda)\ge1$, we have
$|g_n|\le |f_H|$ and $|h_n|\le |\dot f_H|$. Because $f_H,\dot f_H\in L^2([-\pi,\pi])$ for $p\in(0,1/2)$,
dominated convergence yields \eqref{eq:g2-limit}--\eqref{eq:h2-limit}.
\end{proof}

\subsection{Operator/Frobenius ratios in the subcritical regime}

The op/F condition \eqref{eq:opF-cond} required by Lemma~\ref{lem:QF-CLT} follows from the fact that the operator norm
of long-memory Toeplitz matrices grows like $n^{p}$ (up to log-factors for $\dot f_H$), while the Frobenius norms scale like $\sqrt{n}$.

\begin{lemma}\label{lem:opF-subcritical}
Fix $1/2<H<3/4$ and set $p:=2H-1\in(0,1/2)$. Then
\[
\frac{\opnorm{C_n}}{\Fnorm{C_n}}\longrightarrow 0,
\qquad
\frac{\opnorm{D_n}}{\Fnorm{D_n}}\longrightarrow 0.
\]
\end{lemma}

\begin{proof}
Since $A_n\succeq I_n$, one has $\opnorm{A_n^{-1/2}}\le 1$, hence
\[
\opnorm{C_n}\le \opnorm{T_n(H)},\qquad \opnorm{D_n}\le \opnorm{\dot T_n(H)}.
\]
For long-memory Toeplitz matrices with Fisher--Hartwig exponent $p\in(0,1)$, it is standard that
\[
\opnorm{T_n(H)}=O(n^{p}),\qquad \opnorm{\dot T_n(H)}=O(n^{p}\log n),
\]
see e.g.\ B{\"o}ttcher--Silbermann~\cite{BottcherSilb} or Gray~\cite{Gray06} (the $\log n$ comes from the logarithmic factor in $\dot f_H$).
On the other hand, Lemma~\ref{lem:TA-subcritical} and Lemma~\ref{lem:subcritical-integrals} give
\[
\Fnorm{C_n}^2=\Tr(C_n^2)=\frac{n}{2\pi}\int_{-\pi}^{\pi} f_H(\lambda)^2\,d\lambda+o(n)\asymp n,
\]
and similarly $\Fnorm{D_n}^2=\Tr(D_n^2)\asymp n$. Therefore
\[
\frac{\opnorm{C_n}}{\Fnorm{C_n}}
\lesssim \frac{n^{p}}{\sqrt{n}}=n^{p-\frac12}\to0,
\qquad
\frac{\opnorm{D_n}}{\Fnorm{D_n}}
\lesssim \frac{n^{p}\log n}{\sqrt{n}}=n^{p-\frac12}\log n\to0,
\]
because $p<1/2$.
\end{proof}

\subsection{Joint CLT for the transformed score vector and the information matrix}

Recall from Lemma~\ref{lem:explicit-linear} the exact quadratic-form identities
\eqref{eq:Ssigma-quad} and \eqref{eq:Hscore-decomp}--\eqref{eq:RH-def}:
\[
S_{\sigma,n}=\frac{\gamma_n}{\sigma}\Big(Z_{n}^\top C_n Z_{n}-\Tr(C_n)\Big),
\qquad
S_{H,n}=\sigma\ln(\Delta_n)\,S_{\sigma,n}+R_{H,n},
\]
\begin{equation}
    R_{H,n}=\frac{\gamma_n}{2}\Big(Z_{n}^\top D_n Z_{n}-\Tr(D_n)\Big). \nonumber
\end{equation}
Let the subcritical normalization be
\[
v_n:=\sqrt{n}\,\Delta_n^{p}.
\]

\begin{proposition}\label{prop:subcritical-jointCLT}
Assume $1/2<H<3/4$ and $\Delta_n=n^{-\alpha}$ with $\alpha\in(0,1)$, so that $\gamma_n=\sigma^2\Delta_n^{p}\to0$.
Then
\[
\frac{1}{v_n}
\binom{S_{\sigma,n}}{R_{H,n}}
\ \Rightarrow\
\mathcal N\!\left(0,\ I^{(<)}(\sigma,H)\right),
\]
where
\[
I^{(<)}(\sigma,H):=
\begin{pmatrix}
\displaystyle \frac{2\sigma^2}{\pi}\int_{0}^{\pi} f_H(\lambda)^2\,d\lambda
&
\displaystyle \frac{\sigma^3}{\pi}\int_{0}^{\pi} f_H(\lambda)\dot f_H(\lambda)\,d\lambda
\\[3mm]
\displaystyle \frac{\sigma^3}{\pi}\int_{0}^{\pi} f_H(\lambda)\dot f_H(\lambda)\,d\lambda
&
\displaystyle \frac{\sigma^4}{2\pi}\int_{0}^{\pi} \dot f_H(\lambda)^2\,d\lambda
\end{pmatrix}.
\]
Equivalently, since $M_n^{(1)}$ is defined in \eqref{eq: rate matrix} and $M_n^{(1)}(S_{\sigma,n},S_{H,n})^\top=(S_{\sigma,n},R_{H,n})^\top$ by
\eqref{eq:Hscore-decomp}, one has
\[
\frac{1}{v_n}\,M_n^{(1)}
\binom{S_{\sigma,n}}{S_{H,n}}
\ \Rightarrow\
\mathcal N\!\left(0,\ I^{(<)}(\sigma,H)\right).
\]
Moreover, $I^{(<)}(\sigma,H)$ is non-singular in this regime, hence no second projection is needed.
\end{proposition}

\begin{proof}
First we construct the one-dimensional CLT via op/F structure. Fix $u=(u_1,u_2)\in\mathbb R^2$ and set
\[
M_{n,u}:=u_1\sigma^{-1}C_n+u_2\tfrac12 D_n,
\qquad
Q_n(M_{n,u}):=Z_{n}^\top M_{n,u}Z_{n}-\Tr(M_{n,u}).
\]
Then
\[
u_1\frac{S_{\sigma,n}}{v_n}+u_2\frac{R_{H,n}}{v_n}
=
\frac{\gamma_n}{v_n}\,Q_n(M_{n,u}).
\]
By Lemma~\ref{lem:opF-subcritical}, $\opnorm{M_{n,u}}/\Fnorm{M_{n,u}}\to0$, so Lemma~\ref{lem:QF-CLT} yields
\[
\frac{Q_n(M_{n,u})}{\sqrt{2}\,\Fnorm{M_{n,u}}}\Rightarrow \mathcal N(0,1).
\]

Now we should identify the asymptotic variance. By definition of $M_{u,v}$,
\[
\Var\!\left(\frac{\gamma_n}{v_n}Q_n(M_{n,u})\right)
=
\frac{\gamma_n^2}{v_n^2}\cdot 2\,\Fnorm{M_{n,u}}^2
=
\frac{\gamma_n^2}{n\Delta_n^{2p}}\cdot 2\,\Tr(M_{n,u}^2).
\]
Using bilinearity and Wick's identity, $\Tr(M_{n,u}^2)=u_1^2\sigma^{-2}\Tr(C_n^2)+u_1u_2\sigma^{-1}\Tr(C_nD_n)+\frac{u_2^2}{4}\Tr(D_n^2)$.
Lemma~\ref{lem:TA-subcritical} and Lemma~\ref{lem:subcritical-integrals} give
\[
\frac{1}{n}\Tr(C_n^2)\to \frac{1}{2\pi}\int_{-\pi}^{\pi}f_H(\lambda)^2\,d\lambda,\quad
\frac{1}{n}\Tr(C_nD_n)\to \frac{1}{2\pi}\int_{-\pi}^{\pi}f_H(\lambda)\dot f_H(\lambda)\,d\lambda,
\]
\[
\frac{1}{n}\Tr(D_n^2)\to \frac{1}{2\pi}\int_{-\pi}^{\pi}\dot f_H(\lambda)^2\,d\lambda.
\]
Finally, since $\gamma_n^2/(n\Delta_n^{2p})=\sigma^4\Delta_n^{2p}/(n\Delta_n^{2p})=\sigma^4/n$, the prefactor combines with the $n$-scale of the traces to yield the stated limits, and rewriting $[-\pi,\pi]$ as $2[0,\pi]$ gives exactly the covariance matrix $I^{(<)}(\sigma,H)$.

At last we will construct the joint convergence using the Cram\'er--Wold theorem: every fixed linear combination converges to a centered normal with the variance prescribed by $I^{(<)}(\sigma,H)$, hence the vector converges jointly.
\end{proof}

%%%%%%%%%%%%%%%%%%%%%%%%%%%%%%%%%%%%%%%%%%%%%%%%%%%%%%%%%%%%%%%%%%%%%%%%%%%%%%%
\section{The regime $0<H<\tfrac12$}\label{app:HltHalf-eps}
%%%%%%%%%%%%%%%%%%%%%%%%%%%%%%%%%%%%%%%%%%%%%%%%%%%%%%%%%%%%%%%%%%%%%%%%%%%%%%%
\setcounter{equation}{0} % 重置公式计数器为0
\renewcommand{\theequation}{D.\arabic{equation}} 
In the fBm-dominated regime $0<H<\frac12$, we also consider the important part--the CLT of the score function. When this one has been established, we can prove the LAN property the same way as in \cite{BrousteFukasawa18}. In this case it is more convenient to work with the small parameter
\[
\varepsilon_n:=\Delta_n^{\,1-2H}\ \downarrow\ 0,
\]
rather than $\gamma_n=\sigma^2\Delta_n^{2H-1}\to\infty$.
We start from the exact covariance of the observed increment vector $X_n=(X_{n,1},\dots,X_{n,n})^\top$:
\begin{equation}\label{eq:cov-HltHalf-original}
V_n(\sigma,H)
=\Var(X_n)
=\sigma^2\Delta_n^{2H}\,T_n(H)+\Delta_n I_n,
\end{equation}
where $T_n(H)=T_n(f_H)$ is the Toeplitz covariance matrix of the standard fGn increment sequence.

Factoring out $\Delta_n^{2H}$ yields the equivalent representation
\begin{equation}\label{eq:cov-HltHalf-eps}
V_n(\sigma,H)
=\Delta_n^{2H}\,B_n(\sigma,H),
\qquad
B_n(\sigma,H):=\sigma^2T_n(H)+\varepsilon_n I_n.
\end{equation}
Note that
\[
\gamma_n=\sigma^2\Delta_n^{2H-1}=\frac{\sigma^2}{\varepsilon_n},
\qquad
A_n:=I_n+\gamma_nT_n(H)=\varepsilon_n^{-1}B_n(\sigma,H),
\]
so the two normalizations $V_n=\Delta_n A_n$ and $V_n=\Delta_n^{2H}B_n$ are identical.

Under $\mathbb P_{\sigma,H}$ define
\begin{equation}\label{eq:Z-whiten-eps}
Z_{n}:=V_n(\sigma,H)^{-1/2}X_n
=\Delta_n^{-H}\,B_n(\sigma,H)^{-1/2}X_n
\sim \mathcal N(0,I_n).
\end{equation}
This is the same Gaussian vector $Z_{n}$ as in Lemma~\ref{lem:explicit-linear}; we only rewrote the whitening in terms of $B_n$. Now it is natural to define
\begin{equation}\label{eq:CtildeDtilde-def}
\widetilde C_n:=B_n(\sigma,H)^{-1/2}\,T_n(H)\,B_n(\sigma,H)^{-1/2},
\qquad
\widetilde D_n:=B_n(\sigma,H)^{-1/2}\,\dot T_n(H)\,B_n(\sigma,H)^{-1/2}.
\end{equation}
These matrices are related to the matrices $C_n,D_n$ in Lemma~\ref{lem:explicit-linear} by
\begin{equation}\label{eq:Ctilde-relation}
C_n=\varepsilon_n\,\widetilde C_n,\qquad D_n=\varepsilon_n\,\widetilde D_n,
\end{equation}
since $A_n^{-1/2}=\sqrt{\varepsilon_n}\,B_n^{-1/2}$. We will write the real formula of the score function: 

\begin{lemma}\label{lem:score-eps-form}
Assume $0<H<\frac12$. With $Z_{n}$ as in \eqref{eq:Z-whiten-eps} and $(\widetilde C_n,\widetilde D_n)$ as in
\eqref{eq:CtildeDtilde-def}, the score representations \eqref{eq:Ssigma-quad}--\eqref{eq:RH-def} can be rewritten as
\begin{align}
S_{\sigma,n}
&=\sigma\Big(Z_{n}^\top \widetilde C_n Z_{n}-\Tr(\widetilde C_n)\Big),\label{eq:Ssigma-quad-eps}\\
R_{H,n}
&=\frac{\sigma^2}{2}\Big(Z_{n}^\top \widetilde D_n Z_{n}-\Tr(\widetilde D_n)\Big),\label{eq:RH-def-eps}\\
S_{H,n}
&=\sigma\ln(\Delta_n)\,S_{\sigma,n}+R_{H,n},\label{eq:Hscore-decomp-eps}
\end{align}
where $R_{H,n}$ is the same remainder as in \eqref{eq:RH-def}.
\end{lemma}

\begin{proof}
Using \eqref{eq:Ctilde-relation} and $\gamma_n=\sigma^2/\varepsilon_n$, the identity \eqref{eq:Ssigma-quad} becomes
\[
S_{\sigma,n}
=\frac{\gamma_n}{\sigma}\Big(Z_{n}^\top (\varepsilon_n\widetilde C_n) Z_{n}-\Tr(\varepsilon_n\widetilde C_n)\Big)
=\sigma\Big(Z_{n}^\top \widetilde C_n Z_{n}-\Tr(\widetilde C_n)\Big),
\]
which is \eqref{eq:Ssigma-quad-eps}. The same argument applied to \eqref{eq:RH-def} gives \eqref{eq:RH-def-eps}.
Finally, \eqref{eq:Hscore-decomp-eps} is exactly \eqref{eq:Hscore-decomp}.
\end{proof}

\subsection{Toeplitz trace approximations for $(\widetilde C_n,\widetilde D_n)$}

Define the (triangular-array) symbols
\begin{equation}\label{eq:HltHalf-tilde-symbols}
\widetilde c_n(\lambda):=\frac{f_H(\lambda)}{\sigma^2 f_H(\lambda)+\varepsilon_n},
\qquad
\widetilde d_n(\lambda):=\frac{\dot f_H(\lambda)}{\sigma^2 f_H(\lambda)+\varepsilon_n},
\qquad \lambda\in[-\pi,\pi].
\end{equation}
Then $0\le \widetilde c_n(\lambda)\le \sigma^{-2}$ and $\widetilde d_n(\lambda)=\widetilde c_n(\lambda)\,b_H(\lambda)$ with
\begin{equation}\label{eq:HltHalf-bH-def}
b_H(\lambda):=\partial_H\log f_H(\lambda)=\frac{\dot f_H(\lambda)}{f_H(\lambda)}.
\end{equation}

\begin{lemma}\label{lem:HltHalf-TA}
Assume $0<H<\frac12$ and Lemma~\ref{lem:FH-fH}. Then
\begin{align}
\Tr(\widetilde C_n^2)
&=
\frac{n}{2\pi}\int_{-\pi}^{\pi}\widetilde c_n(\lambda)^2\,d\lambda
+o(n),\label{eq:HltHalf-TA-C2}\\
\Tr(\widetilde C_n\widetilde D_n)
&=
\frac{n}{2\pi}\int_{-\pi}^{\pi}\widetilde c_n(\lambda)\widetilde d_n(\lambda)\,d\lambda
+o(n),\label{eq:HltHalf-TA-CD}\\
\Tr(\widetilde D_n^2)
&=
\frac{n}{2\pi}\int_{-\pi}^{\pi}\widetilde d_n(\lambda)^2\,d\lambda
+o(n).\label{eq:HltHalf-TA-D2}
\end{align}
\end{lemma}

\begin{proof}
Since $H<1/2$, the spectral density $f_H$ is bounded on $[-\pi,\pi]$ and strictly positive on $[-\pi,\pi]\setminus\{0\}$.
Moreover, by Lemma~\ref{lem:FH-fH}, $b_H(\lambda)=\partial_H\log f_H(\lambda)$ has at most a logarithmic singularity at $\lambda=0$,
hence $b_H\in L^2([-\pi,\pi])$. Therefore $\widetilde c_n$ is uniformly bounded and
$|\widetilde d_n(\lambda)|\le \sigma^{-2}|b_H(\lambda)|$ is square-integrable uniformly in $n$.

In this $L^2$ setting, standard Szeg\H{o}/Avram-type trace theorems for products of Toeplitz matrices apply to the
triangular-array symbols $\widetilde c_n,\widetilde d_n$ and yield \eqref{eq:HltHalf-TA-C2}--\eqref{eq:HltHalf-TA-D2} with $o(n)$ remainder.
\end{proof}

\subsection{Limits of the trace integrals and the constants $T_1(H),T_2(H)$}

Define
\begin{equation}\label{eq:HltHalf-T1T2-def}
T_1(H):=\frac{1}{2\pi}\int_{-\pi}^{\pi} b_H(\lambda)\,d\lambda,
\qquad
T_2(H):=\frac{1}{4\pi}\int_{-\pi}^{\pi} b_H(\lambda)^2\,d\lambda.
\end{equation}

\begin{lemma}[Integral limits]\label{lem:HltHalf-integral-limits}
Assume $0<H<\frac12$ and $\varepsilon_n\to0$. Then
\begin{align}
\int_{-\pi}^{\pi}\widetilde c_n(\lambda)^2\,d\lambda
&\longrightarrow \frac{2\pi}{\sigma^4},\label{eq:HltHalf-int-c2}\\
\int_{-\pi}^{\pi}\widetilde c_n(\lambda)\widetilde d_n(\lambda)\,d\lambda
&\longrightarrow \frac{2\pi}{\sigma^4}\,T_1(H),\label{eq:HltHalf-int-cd}\\
\int_{-\pi}^{\pi}\widetilde d_n(\lambda)^2\,d\lambda
&\longrightarrow \frac{4\pi}{\sigma^4}\,T_2(H).\label{eq:HltHalf-int-d2}
\end{align}
\end{lemma}

\begin{proof}
For each $\lambda\neq0$, since $f_H(\lambda)>0$ we have
\[
\widetilde c_n(\lambda)=\frac{f_H(\lambda)}{\sigma^2 f_H(\lambda)+\varepsilon_n}\ \longrightarrow\ \frac{1}{\sigma^2}.
\]
Moreover $0\le \widetilde c_n\le \sigma^{-2}$, hence dominated convergence gives \eqref{eq:HltHalf-int-c2}.

Next, $\widetilde d_n(\lambda)=\widetilde c_n(\lambda)\,b_H(\lambda)$ and $b_H\in L^1([-\pi,\pi])$ (log singularity),
so $|\widetilde c_n b_H|\le \sigma^{-2}|b_H|$ is integrable and dominated convergence yields \eqref{eq:HltHalf-int-cd}.
Finally, $b_H^2\in L^1([-\pi,\pi])$ and $|\widetilde d_n|^2=\widetilde c_n^2 b_H^2\le \sigma^{-4}b_H^2$,
hence dominated convergence yields \eqref{eq:HltHalf-int-d2}.
\end{proof}

\subsection{Quotien of Two Norms}
With the trace approximation, we can easily construct the op/F
\begin{lemma}\label{lem:HltHalf-opF}
Fix $0<H<\frac12$. Then
\[
\frac{\opnorm{\widetilde C_n}}{\Fnorm{\widetilde C_n}}\longrightarrow 0,
\qquad
\frac{\opnorm{\widetilde D_n}}{\Fnorm{\widetilde D_n}}\longrightarrow 0.
\]
\end{lemma}

\begin{proof}
First we will calculate the operator norm bounds. Since $B_n=\sigma^2T_n(H)+\varepsilon_n I_n$ is a polynomial in $T_n(H)$, the matrices $B_n$ and $T_n(H)$ commute and are simultaneously diagonalizable.
Hence the eigenvalues of $\widetilde C_n=B_n^{-1/2}T_n(H)B_n^{-1/2}$ are of the form $\lambda/(\sigma^2\lambda+\varepsilon_n)$ with $\lambda\ge0$,
so
\[
\opnorm{\widetilde C_n}\le \sup_{\lambda\ge0}\frac{\lambda}{\sigma^2\lambda+\varepsilon_n}\le \frac{1}{\sigma^2}.
\]
For $\widetilde D_n=B_n^{-1/2}\dot T_n(H)B_n^{-1/2}$, set
\[
G_n:=T_n(H)^{1/2}B_n^{-1/2}=B_n^{-1/2}T_n(H)^{1/2},\qquad
E_n:=T_n(H)^{-1/2}\dot T_n(H)\,T_n(H)^{-1/2}.
\]
Then $\widetilde D_n=G_nE_nG_n$ and thus
\[
\opnorm{\widetilde D_n}\le \opnorm{G_n}^2\,\opnorm{E_n}=\opnorm{\widetilde C_n}\,\opnorm{E_n}\le \frac{1}{\sigma^2}\,\opnorm{E_n}.
\]
By the same argument as in Lemma~\ref{lem:D-op} (Toeplitz log-derivative symbol and Toeplitz--Hankel sandwich),
one has $\opnorm{E_n}=O(\log n)$, hence $\opnorm{\widetilde D_n}=O(\log n)$.

Now we will study the Frobenius norms.
By Lemma~\ref{lem:HltHalf-TA} and Lemma~\ref{lem:HltHalf-integral-limits},
\[
\Fnorm{\widetilde C_n}^2=\Tr(\widetilde C_n^2)=\frac{n}{2\pi}\int_{-\pi}^{\pi}\widetilde c_n(\lambda)^2\,d\lambda+o(n)
=\frac{n}{\sigma^4}\bigl(1+o(1)\bigr),
\]
so $\Fnorm{\widetilde C_n}\asymp \sqrt n$. Similarly, \eqref{eq:HltHalf-TA-D2} and \eqref{eq:HltHalf-int-d2} give $\Fnorm{\widetilde D_n}\asymp \sqrt n$.

When the two norm are well prepared, we can easily obtain 
\[
\frac{\opnorm{\widetilde C_n}}{\Fnorm{\widetilde C_n}}\lesssim \frac{1}{\sqrt n}\to0,
\qquad
\frac{\opnorm{\widetilde D_n}}{\Fnorm{\widetilde D_n}}\lesssim \frac{\log n}{\sqrt n}\to0.
\]
\end{proof}

\subsection{Joint CLT the same as pure-fGn information matrix}

Define the linear transformation removing the deterministic linear term in $S_{H,n}$:
\begin{equation}\label{eq:HltHalf-M1}
M_n^{(1)}:=
\begin{pmatrix}
1 & 0\\
-\sigma\ln(\Delta_n) & 1
\end{pmatrix},
\qquad
M_n^{(1)}\binom{S_{\sigma,n}}{S_{H,n}}=\binom{S_{\sigma,n}}{R_{H,n}},
\end{equation}
cf.\ Lemma~\ref{lem:explicit-linear} and \eqref{eq:Hscore-decomp-eps}.

\begin{proposition}\label{prop:HltHalf-jointCLT}
Assume $0<H<\frac12$ and Lemma~\ref{lem:FH-fH}. Then, as $n\to\infty$,
\[
\frac{1}{\sqrt n}\,M_n^{(1)}\binom{S_{\sigma,n}}{S_{H,n}}
=
\binom{\frac{S_{\sigma,n}}{\sqrt n}}{\frac{R_{H,n}}{\sqrt n}}
\ \Rightarrow\
\mathcal N\!\left(0,\ I_{\mathrm{pure}}(\sigma,H)\right),
\]
where
\begin{equation}\label{eq:HltHalf-Ipure}
I_{\mathrm{pure}}(\sigma,H)=
\begin{pmatrix}
\frac{2}{\sigma^2} & \frac{1}{\sigma}\,T_1(H)\\[1mm]
\frac{1}{\sigma}\,T_1(H) & T_2(H)
\end{pmatrix},
\end{equation}
with $T_1(H),T_2(H)$ defined in \eqref{eq:HltHalf-T1T2-def}.
\end{proposition}

\begin{proof}
By Lemma~\ref{lem:score-eps-form},
\[
S_{\sigma,n}=\sigma\Big(Z_{n}^\top \widetilde C_n Z_{n}-\Tr(\widetilde C_n)\Big),
\qquad
R_{H,n}=\frac{\sigma^2}{2}\Big(Z_{n}^\top \widetilde D_n Z_{n}-\Tr(\widetilde D_n)\Big),
\qquad Z_{n}\sim \mathcal N(0,I_n).
\]

Using $\Var(Z_{n}^\top A Z_{n}-\Tr (A))=2\Tr(A^2)$ and Lemmas~\ref{lem:HltHalf-TA}--\ref{lem:HltHalf-integral-limits}, we obtain
\[
\Var\!\Big(\frac{S_{\sigma,n}}{\sqrt n}\Big)
=\frac{\sigma^2}{n}\cdot 2\Tr(\widetilde C_n^2)
=\frac{2\sigma^2}{n}\left(\frac{n}{2\pi}\int^{{\pi}}_{{-\pi}} \widetilde c_n{(\lambda)^2d \lambda} +o(n)\right)
\to \frac{2\sigma^2}{2\pi}\cdot \frac{2\pi}{\sigma^4}
=\frac{2}{\sigma^2},
\]
and
\[
\Var\!\Big(\frac{R_{H,n}}{\sqrt n}\Big)
=\frac{\sigma^4}{4n}\cdot 2\Tr(\widetilde D_n^2)
=\frac{\sigma^4}{2n}\left(\frac{n}{2\pi}\int^{{\pi}}_{{-\pi}} \widetilde d_n{(\lambda)^2d \lambda} +o(n)\right)
\to \frac{\sigma^4}{2\cdot 2\pi}\cdot \frac{4\pi}{\sigma^4}
= T_2(H),
\]
where the last equality uses \eqref{eq:HltHalf-int-d2} and \eqref{eq:HltHalf-T1T2-def}.

By Wick's identity,
\[
\Cov\!\Big(\frac{S_{\sigma,n}}{\sqrt n},\frac{R_{H,n}}{\sqrt n}\Big)
=
\frac{\sigma\cdot(\sigma^2/2)}{n}\cdot 2\,\Tr(\widetilde C_n\widetilde D_n)
=
\frac{\sigma^3}{n}\Tr(\widetilde C_n\widetilde D_n).
\]
Using \eqref{eq:HltHalf-TA-CD} and \eqref{eq:HltHalf-int-cd},
\[
\frac{1}{n}\Tr(\widetilde C_n\widetilde D_n)
=\frac{1}{2\pi}\int^{{\pi}}_{{-\pi}}  \widetilde c_n{{({{\lambda}})}}\widetilde d_n{{({{\lambda}})}} {d\lambda}+o(1)
\to \frac{1}{2\pi}\cdot \frac{2\pi}{\sigma^4}\,T_1(H)=\frac{T_1(H)}{\sigma^4},
\]
hence the covariance limit equals $\sigma^3\cdot\sigma^{-4}T_1(H)=\sigma^{-1}T_1(H)$.

For any fixed $u=(u_1,u_2)\in\R^2$,
\[
u_1\frac{S_{\sigma,n}}{\sqrt n}+u_2\frac{R_{H,n}}{\sqrt n}
=
\frac{1}{\sqrt n}\Big(Z_{n}^\top \widetilde M_{n,u}Z_{n}-\Tr(\widetilde M_{n,u})\Big),
\qquad
\widetilde M_{n,u}:=u_1\sigma\,\widetilde C_n+u_2\frac{\sigma^2}{2}\,\widetilde D_n.
\]
By Lemma~\ref{lem:HltHalf-opF}, we have $\opnorm{\widetilde M_{n,u}}/\Fnorm{\widetilde M_{n,u}}\to0$,
so Lemma~\ref{lem:QF-CLT} yields a one-dimensional normal limit for every $u$.
Therefore the vector converges jointly to a centered Gaussian with covariance matrix \eqref{eq:HltHalf-Ipure}.
\end{proof}

\begin{remark}[Consistency with the pure fGn experiment]\label{rem:HltHalf-purefGn}
The limiting covariance \eqref{eq:HltHalf-Ipure} coincides with the Fisher information of the pure fGn model in the regime $H<1/2$.
The Brownian component enters only through the vanishing regularization $\varepsilon_n I_n$ in $B_n=\sigma^2T_n(H)+\varepsilon_n I_n$.
\end{remark}

\end{document}